\documentclass{article}

\usepackage{amsmath}
\usepackage{mathrsfs}
\usepackage{amssymb}
\usepackage{amsthm}

\usepackage{graphicx}

\usepackage{enumitem}
\usepackage{subfigure}
\usepackage{color}
\usepackage{multirow}

\usepackage{verbatim}
\usepackage{algorithm,algorithmicx,algpseudocode}
\usepackage{cite}
\usepackage{bbm}
\usepackage{bigstrut}
\usepackage[left=1in,top=1in,right=1in,bottom=1in,letterpaper]{geometry}
\usepackage[linktocpage,colorlinks,linkcolor=blue,anchorcolor=blue,citecolor=blue,urlcolor=blue]{hyperref}
\usepackage{cleveref}

\usepackage{tikz}

\makeatletter
\newcommand\footnoteref[1]{\protected@xdef\@thefnmark{\ref{#1}}\@footnotemark}
\makeatother
\crefname{figure}{Figure}{Figures}

\numberwithin{equation}{section}
\crefname{equation}{equation}{equations}
\crefname{section}{Section}{Sections}
\newtheorem{theorem}{Theorem}[section]
\crefname{theorem}{Theorem}{Theorems}
\crefname{definition}{Definition}{Definitions}

\newtheorem{corollary}[theorem]{Corollary}
\newtheorem{lemma}[theorem]{Lemma}
\crefname{lemma}{Lemma}{Lemmas}

\newtheorem{proposition}[theorem]{Proposition}
\crefname{proposition}{Proposition}{Propositions}

\newtheorem{definition}[theorem]{Definition}
\newtheorem{fact}[theorem]{Fact}
\crefname{fact}{Fact}{Facts}
\crefname{corollary}{Corollary}{Corollaries}

\newtheorem{assumption}{Assumption}
\crefname{assumption}{Assumption}{Assumptions}
\crefname{algorithm}{Algorithm}{Algorithms}

\newcommand{\G}{\mathcal{G}}

\newcommand{\cT}{\mathcal{T}} 

\newcommand{\M}{\mathcal{M}}
\newcommand{\N}{\mathcal{N}}

\newcommand{\cM}{\mathcal{M}}

\newcommand{\cX}{\mathcal{X}}

\newcommand{\bG}{\mathbf{G}}

\newcommand{\U}{\mathbf{u}}
\newcommand{\V}{\mathbf{v}}

\newcommand{\R}{\mathbb{R}}
\newcommand{\E}{\mathbb{E}}

\newcommand{\half}{\frac{1}{2}}

\newcommand{\X}{\mathbf{x}}
\newcommand{\Y}{\mathbf{y}}

\newcommand{\be}{\begin{equation}}
\newcommand{\ee}{\end{equation}}
\newcommand{\ba}{\begin{array}}
\newcommand{\ea}{\end{array}}
\newcommand{\bad}{\begin{aligned}}
\newcommand{\ead}{\end{aligned}}

\newcommand{\normtwo}[1]{\| #1 \|_2}

\newcommand{\normfro}[1]{\| #1 \|_{\text{F}}}
\newcommand{\normfroinf}[1]{\| #1 \|_{\text{F},\infty}}

\newcommand{\inp}[2]{\left\langle #1, #2 \right\rangle}
\newcommand{\argmin}{\mathop{\rm argmin}}

\newcommand{\p}{\mathcal{P}}
\newcommand{\ps}{\mathcal{P_{\St}}}

\newcommand{\dist}{\mathrm{dist}}

\newcommand{\Retr}{\mathcal{R}}

\newcommand{\Tr}{\mathrm{Tr}}

\newcommand{\T}{\mathrm{T}}
 
\newcommand{\st}{\mathrm{s.t. }}
\newcommand{\ie}{\mathrm{i.e. }}

\newcommand{\St}{\mathrm{St}}

 \newcommand{\grad}{\mathrm{grad}}
 \newcommand{\Hess}{\mathrm{Hess}}

\begin{document}

\title{Decentralized Riemannian Gradient Descent on the Stiefel Manifold}
\date{}



\author{Shixiang Chen \thanks{$^{1}$The Wm Michael Barnes '64 Department of Industrial and Systems Engineering, Texas A\&M University, College Station, TX 77843. 
		Email addresses: {\tt\small sxchen@tamu.edu} (S. Chen), {\tt\small alfredo.garcia@tamu.edu } (A. Garcia), {\tt\small shahin@tamu.edu} (S.  Shahrampour).}   \and Alfredo Garcia\footnotemark[1] \and Mingyi Hong	\thanks{$^{2}$The Department of Electrical and Computer Engineering, University of Minnesota, Minneapolis, MN 55455.
		Email address: {\tt\small mhong@umn.edu} (M. Hong).}  \and Shahin Shahrampour\footnotemark[1] 

}


\maketitle
\begin{abstract}
We consider a distributed non-convex optimization where a network of agents aims at minimizing a global function over the Stiefel manifold. The global function is represented as a finite sum of smooth local functions, where each local function is associated with one agent and agents communicate with each other over an undirected connected graph. The problem is non-convex as local functions are possibly non-convex (but smooth) and the Steifel manifold is a non-convex set. We present a decentralized Riemannian stochastic gradient method (DRSGD) with the convergence rate of $\mathcal{O}(1/\sqrt{K})$ to a stationary point. To have exact convergence with constant stepsize, we also propose a decentralized Riemannian gradient tracking algorithm (DRGTA) with the convergence rate of $\mathcal{O}(1/K)$ to a stationary point. We use multi-step consensus to preserve the iteration in the local (consensus) region. DRGTA is the first decentralized algorithm with exact convergence for distributed   optimization on Stiefel manifold.
\end{abstract}

	\section{Introduction}
Distributed optimization  has received significant attention in the past few years in   machine learning, control and signal processing. There are mainly two scenarios where distributed algorithms are  necessary: (i) the data is geographically distributed over networks and/or (ii) the computation on a single (centralized) server is too expensive (large-scale data setting). 
In this paper, we  consider  the following  multi-agent optimization problem 
\be   \label{opt_problem}
\bad
\min \frac{1}{n}\sum_{i=1}^n &f_i(x_i)\\
\st \quad  x_1= x_2 =&\ldots = x_n,\\
x_i\in\M,&\quad \forall i = 1,\ldots,n,
\ead
\ee
where $f_i$ has $L-$Lipschitz continuous gradient in Euclidean space and $\M:=\St(d,r)= \{ x\in\R^{d\times r}: x^\top x = I_r \}$ is the Stiefel manifold. Unlike the Euclidean distributed setting, problem \eqref{opt_problem}  is defined on the Stiefel manifold, which is a non-convex set. Many important applications can be written in the form \eqref{opt_problem}, e.g., decentralized spectral analysis  \cite{kempe2008decentralized,gang2021linearly,huang2020communication}, dictionary learning \cite{raja2015cloud},  eigenvalue estimation of the covariance matrix \cite{penna2014decentralized} in wireless sensor networks, and deep neural networks with orthogonal constraint \cite{arjovsky2016unitary,vorontsov2017orthogonality,huang2018orthogonal}.  

Problem \eqref{opt_problem} can generally represent a risk minimization. One approach to solving \eqref{opt_problem} is collecting all variables to a central server and running a centralized algorithm. However, when the dataset is massive (or the data dimension is large), this causes memory issues and computational burden on the central server. Then, it is more efficient to take a decentralized approach and use local computation based on a network topology. In this case, each local function $f_i$ is associated with one agent in the network, and agents communicate with each other over an undirected connected graph. For example, for stochastic gradient descent (SGD), \cite{lian2017can} show that the decentralized SGD can be faster than centralized SGD, especially when training neural networks. More importantly, a central server may not exist in practice.

\subsection{Our Contributions}
In this paper, we focus on the decentralized setting and design efficient decentralized algorithms to solve \eqref{opt_problem} over any connected undirected network. Our contributions are as follows:\\
	(1)  We show the convergence of the decentralized stochastic Riemannian gradient method (\cref{alg:DRPG}) for solving \eqref{opt_problem}.  Specifically, the iteration complexity of obtaining an $\epsilon-$stationary point (see \cref{def:stationary}) is $\mathcal{O}(1/\epsilon^2)$ in expectation \footnote{\label{note1} We have omitted the dependence on network parameters here.}. \\
	(2)  To achieve exact convergence with constant stepsize, we propose a  gradient tracking algorithm (DRGTA) (\cref{alg:DRG_GT}) for solving \eqref{opt_problem}. For DRGTA, the iteration complexity of obtaining an $\epsilon-$stationary point is $\mathcal{O}(1/\epsilon)$ \footnoteref{note1}.
	
	Importantly, both of the proposed algorithms are retraction-based and DRGTA is vector transport-free. These two features make the algorithms computationally cheap and conceptually simple. DRGTA is the first decentralized algorithm with exact convergence for distributed optimization on the Stiefel manifold.

\subsection{Related works}
Decentralized optimization   has been well-studied in Euclidean space. The decentralized (sub)-gradient methods were studied in \cite{tsitsiklis1986distributed,nedic2010constrained,yuan2016convergence,chen2020distributed} and a distributed dual averaging subgradient method was proposed in \cite{duchi2011dual}. However, with a constant stepsize $\beta>0,$ these methods can only converge to a $\mathcal{O}(\frac{\beta}{1-\sigma_2})-$neighborhood of a stationary point, where $\sigma_2$ is a network parameter (see \cref{assump:doubly-stochastic}). To achieve exact convergence with a fixed stepsize, gradient tracking algorithms were proposed in  \cite{shi2015extra,xu2015augmented,di2016next,qu2017harnessing,nedic2017achieving,yuan2018exact},  to name a few. The convergence analysis can be unified via a primal-dual framework \cite{alghunaim2019decentralized}.  Another way to use the constant stepsize is decentralized ADMM and its variants \cite{mota2013d,chang2014multi,shi2014linear,aybat2017distributed}. 
Also, decentralized stochastic gradient method for non-convex smooth problems were well-studied in \cite{lian2017can,assran2019stochastic,xin2020near}, etc. We refer to the survey paper \cite{nedic2018network} for a complete review on the state-of-the-art algorithms and the role of network topology.

The problem \eqref{opt_problem} can be thought as a constrained decentralized problem in Euclidean space, but since the Stiefel manifold constraint is non-convex, none of the above works can solve the problem. On the other hand, we can also treat \eqref{opt_problem} as a smooth problem over the Stiefel manifold. However, the  constraint $x_1=x_2=\ldots=x_n$ is difficult to handle due to the lack of linearity on $\M$.  Since the Stiefel manifold is an embedded submanifold in Euclidean space, our viewpoint is  to treat the problem in Euclidean space and develop new tools based on Riemannian manifold optimization \cite{edelman1998geometry,Absil2009,boumal2019global}. For the optimization problem \eqref{opt_problem}, a decentralized Riemannian gradient tracking algorithm was presented in \cite{shah2017distributed}. The vector transport operation should be used in \cite{shah2017distributed}, which brings not only expensive computation but also analysis difficulties. Moreover, they need to use asymptotically infinite number of consensus steps.   Other distributed algorithms were either specifically designed for the PCA problem \cite{penna2014decentralized,raja2015cloud,gang2021linearly} or  in centralized topology \cite{fan2019distributed,huang2020communication,wang2020distributed}.  For these decentralized algorithms, diminishing stepsize or asymptotically infinite number of communication steps should be utilized to get exact solution. Different from all these works, DRGTA requires a {\it finite} number of communications using a {\it constant} step-size.

As a special case of problem \eqref{opt_problem}, the Riemannian consensus problem is well-studied; see \cite{sarlette2009consensus,tron2012riemannian,markdahl2020high,chen2020consensus}.  Recently, it was shown in \cite{chen2020consensus} that the multi-step consensus algorithm (DRCS) converges linearly to the global consensus in a local region. 
\begin{definition}[Consensus]\label{def:consensus}
	Consensus is the configuration where $x_i = x_j\in\M$ for all $i,j\in[n]$. We define the consensus set as follows
	\be\label{def:consensus set}
	\cX^*:=\{ \X\in \M^n: x_1=x_2=\ldots=x_n \}.
	\ee
\end{definition}
Specifically, DRCS iterates $\{\X_k\}$ have the following convergence property in a neighborhood of $\cX^*$
\be\label{ineq:consensus_linear_rate}
\dist(\X_{k+1}, \cX^*) \leq \vartheta\cdot\dist(\X_{k}, \cX^*), \quad \vartheta \in(0,1), 
\ee
where $\dist^2(\X,\cX^*): =  \min_{y\in\M} \frac{1}{n} \sum_{i=1}^n\normfro{y-x_{i}}^2$ and $\X^\top = ( x_1^\top \ x_2^\top \ \ldots \ x_n^\top )$.
 The linear rate of DRCS sheds some lights on designing the decentralized Riemannain gradient method on Stiefel manifold. 
More details will be provided in \cref{sec:consensus on stiefel}.

\section{Preliminaries}\label{sec:preliminary}
{\bf Notation:} The undirected connected graph $G=(\mathcal{V},\mathcal{E})$ is composed of $|\mathcal{V}|=n$ agents. We use $\X$ to denote a collection of all local variables $x_i$ by stacking them, i.e., $\X^\top = ( x_1^\top \ x_2^\top \ \ldots \ x_n^\top )$. For $\M$, the $n-$fold Cartesian product  of $\M$ with itself  is denoted as $\M^n = \M\times\ldots\times\M$. We use $[n]: = \{1,2,\ldots,n\}$. For $\X\in\M^n$, we denote the $i-$th block by $[\X]_i=x_i$. 
We denote the tangent space of $\M$ at point $x$ as $\T_x\M$ and the normal space as $N_x\M$. The inner product on $\T_x\M$ is induced from the Euclidean inner product $\inp{x}{y}=\Tr(x^\top y)$. Denote $\normfro{\cdot}$ as the Frobenius norm and $\normtwo{\cdot}$ as the operator norm. The Euclidean gradient of function $g(x)$ is $\nabla g(x)$ and the Riemannain gradient is $\grad g(x).$ Let $I_r$ and $0_r$ be the $r\times r$ identity matrix and zero matrix, respectively. And let $\textbf{1}_n$ denote the $n$ dimensional vector with all ones.

The network structure is modeled using a matrix,   denoted by $W$, which satisfies the following assumption.
\begin{assumption}\label{assump:doubly-stochastic}
	We assume that the undirected graph  $G$ is    connected and $W$ is doubly stochastic, i.e.,
	(i) $W=W^\top$; (ii)  $W_{ij}\geq 0$ and $1>W_{ii}>0$ for all $i,j;$ (iii)  Eigenvalues of $W$ lie in $(-1,1]$. The second largest singular value $\sigma_2$ of $W$ lies in $\sigma_2\in[0,1)$.
\end{assumption}

We now introduce some preliminaries of Riemannian manifold and fundamental lemmas. 

\subsection{Induced Arithmetic Mean}
Denote the Euclidean average point of $x_1,\ldots, x_n$ by 
\be\label{def:Euclidean mean}
\hat x := \frac{1}{n}\sum_{i=1}^n x_{i}.
\ee 
To measure the degree of consensus, the error $\normfro{x_i-\hat x}$ is typically used in the Euclidean decentralized algorithms. Instead, here we use the induced arithmetic mean(IAM) \cite{sarlette2009consensus} on $\St(d,r)$, defined as follows 
  
\begin{align} 
\bar{x}& := \argmin_{y\in\St(d,r)} \sum_{i=1}^n\normfro{y-x_{i}}^2 = \ps(\hat x),\label{eq:IAM}\tag{IAM}
\end{align}
where $\p_{\St}(\cdot)$ is the orthogonal projection onto $\St(d,r)$. 
Define
\be\label{def:IAM}
\bar\X = \mathbf{1}_n\otimes \bar x.
\ee
 Then the distance between $\X$ and $\cX^*$ is given by
\[ \dist^2(\X, \cX^*)= \min_{y\in\St(d,r)} \frac{1}{n} \sum_{i=1}^n\normfro{y-x_{i}}^2 = \frac{1}{n}\normfro{\X-\bar\X}^2.  \]
Furthermore, we define the $l_{F,\infty}$ distance  between $\X$ and $\bar \X$ as  
\be\label{distance_infty}
\normfroinf{\X-\bar\X} = \max_{i\in[n]} \normfro{x_i-\bar{x}}. \tag{$l_{F,\infty}$}
\ee
We will develop the analysis of decentralized Riemannian gradient descent by studying the error distance $\normfro{\X-\bar\X}$ and $\normfroinf{\X-\bar\X}$.
\subsection{Optimality Condition }
Next, we introduce the optimality condition on manifold $\cM.$ 
Consider the following centralized optimization problem over a matrix manifold $\M$ \be\label{prob:central} \min h(x) \quad \st \quad  x\in\M. \ee
Since we use the metric on tangent space $\T_x\M$ induced from the Euclidean inner product $\inp{\cdot}{\cdot}$, the Riemannian gradient $\grad h(x)$ on $\St(d,r)$ is given by $\grad h(x) = \p_{\T_{x\M}}\nabla h(x)$, where $\p_{\T_{x\M}}$ is the orthogonal projection onto $\T_x\M$. More specifically, we have \[\p_{\T_{x\M}}y = y - \frac{1}{2} x(x^\top y + y^\top x)\] for any $y\in\R^{d\times r}$; see \cite{edelman1998geometry,Absil2009}. 
The  necessary first-order optimality condition of problem \eqref{prob:central} is given as follows.  
\begin{proposition}\cite{Yang-manifold-optimality-2014,boumal2019global}\label{prop:optcond}
	Let $x \in \M $ be a local optimum for \eqref{prob:central}. If $h$
	is differentiable at $x$, then $\grad h(x) = 0$.
\end{proposition} 
Therefore, $x$ is a first-order critical point (or critical point) if $\grad h(x)=0$. 
Let $\bar x$ be the IAM of $\X$. We  define the $\epsilon-$ stationary point of problem \eqref{opt_problem} as follows. 
\begin{definition}[$\epsilon$-Stationarity]\label{def:stationary}
	We say that $\X^\top=(x_1^\top \ x_2^\top \ \ldots\ x_n^\top)$ is an $\epsilon-$ stationary point of problem \eqref{opt_problem} if the following holds:
	\[ \frac{1}{n}\sum_{i=1}^{n}\normfro{x_i-\bar x}^2 \leq \epsilon \quad \forall i,j\in[n] \]
	and 
	\[ \normfro{\grad f(\bar x)}^2 \leq \epsilon, \] where we use the notation $f(\bar x) = \frac{1}{n} \sum_{i=1}^n f_i(\bar x)$. 
	
\end{definition}

\subsection{Basic Lemmas}\label{subsec:basic lemmas}
Our goal is to  develop the decentralized version of centralized Riemannian gradient descent on $\St(d,r).$ 
Simply speaking, the centralized Riemannian gradient descent \cite{Absil2009,boumal2019global}  iterates as \[ x_{k+1} = \Retr_{x_k}(-\alpha\grad h(x_k)),\] i.e.,  updating along a negative Riemannian gradient direction on the tangent space, then performing a operation called \textit{ retraction} $ \Retr_{x_k}$ to ensure feasibility.  We use the definition  of retraction in \cite[Definition 1]{boumal2019global}. The retraction  is the relaxation of exponential mapping, and more importantly, it is computationally cheaper. 
We also assume the  second-order boundedness of  retraction.  It means that \[ \Retr_{x}(\xi) = x+\xi + \mathcal{O}(\normfro{\xi}^2).\] That is, $\Retr_{x}(\xi)$ is locally  good approximation to $x+\xi$. Such   kind of approximation is well enough to take the place of exponential map for the first-order algorithms.  
\begin{lemma} \cite{boumal2019global,Liu-So-Wu-2018}\label{lem:nonexpansive_bound_retraction}
	Let $\Retr$ be a second-order retraction over $\St(d,r)$, we have
	\be\label{ineq:ret_second-order}\bad
	 \normfro{ \Retr_x(\xi)& - (x+\xi)} \leq M\normfro{\xi}^2, \\ & \forall x\in\St(d,r), 
	 \forall \xi\in\T_x\M. \ead\tag{P1} \ee
	Moreover, if the retraction is the polar decomposition. 
	For all $\X\in\St(d,r)$ and
	$\xi\in\T_x\M$, the following  inequality  holds for any $y\in \St(d,r)$ \cite[Lemma 1]{li2019nonsmooth}:
	\be\label{ineq:ret_nonexpansive}
	\normfro{\Retr_{x}(\xi)-y}\leq \normfro{x+\xi-y} .
	\ee
\end{lemma}
  {In the sequel, {\it retraction} refers to the {\it polar retraction} to present a simple analysis, unless otherwise noted. More details  on the polar retraction is provided in \cref{sec:append_retraction}. }  Throughout the paper, we assume that every $f_i(x)$   is Lipschitz smooth. 
\begin{assumption}\label{assump:lips in Rn}
 Each $f_i(x)$ has $L-$Lipschitz continuous gradient, and let $D:=\max_{x\in\St(d,r)}\normfro{\nabla f_i(x) }.$ Therefore, $\nabla f(x)$ is also $L$-Lipschitz continuous and  $D =\max_{x\in\St(d,r)}\normfro{\nabla f(x) }.$
\end{assumption}
 We have two similar Lipschitz continuous inequalities on Stiefel manifold as the Euclidean-type ones \cite{nesterov2013introductory}. We provide the proof in Appendix.

\begin{lemma}[Lipschitz-type inequalities] \label{lem:lipschitz}
	For any $x,y\in\St(n,d)$ and $\xi\in\T_{x}\M$, if $f(x)$ is $L-$Lipschitz smooth in Euclidean space, then  there exists a constant $L_g=  L + L_n$ such that
	\be\label{ineq:lipschitz}
	\left|f(y) - \left[ f(x) + \inp{\grad f(x)}{y-x} \right]  \right|\leq \frac{L_g}{2}\normfro{y-x}^2,
	\ee
	where $L_n= \max_{x\in\St(d,r)} \normtwo{  \nabla f(x)}$. 
	Moreover, define  $L_G=L+2L_n$, one has
	\be\label{ineq:lips_riemanniangrad} \normfro{\grad f(x)-\grad f(y)}\leq L_G\normfro{y-x}.  \ee
\end{lemma}
The difference between two Riemannian gradients is not well-defined on general manifold. However, since the Stiefel manifold is embedded in Euclidean space, we are free to do so.   Another similar inequality as \eqref{ineq:lipschitz} is the  restricted Lipschitz-type gradient presented in \cite[Lemma 4]{boumal2019global}.
But they do not provide an inequality as \eqref{ineq:lips_riemanniangrad}.  One could also consider the following Lipschitz  inequality (see \cite{zhang2016first,Absil2009})
\[ \normfro{\mathrm{P}_{x\rightarrow y}\grad f(x) - \grad f(y)}\leq L_g^\prime d_g(x,y), \]
where $\mathrm{P}_{x\rightarrow y}:\T_x\M \rightarrow \T_y\M$ is the vector transport and $d_g(x,y)$ is the geodesic distance. Since involving vector transport and geodesic distance   brings computational and conceptual difficulties, we choose to use the form of \eqref{ineq:lips_riemanniangrad}  for simplicity.  In fact, $L_g$, $\tilde{L}_g$ and $L_g^\prime$ are  the same up to a constant. A detailed comparison is provided in  \cref{subsection:append-comparison lipschitz ineq}.  

We will use \cref{lem:nonexpansive_bound_retraction} and \cref{lem:lipschitz} to present a parallel analysis to the decentralized Euclidean gradient methods\cite{nedic2010constrained,nedic2017achieving,lian2017can}.

\section{Review of consensus on Stiefel manifold}\label{sec:consensus on stiefel}
 For the decentralized gradient-type algorithms  \cite{tsitsiklis1986distributed, nedic2010constrained,yuan2016convergence,shi2015extra,nedic2017achieving,lian2017can}, they are based  on the linear convergence of consensus iteration  in Euclidean space.

 The consensus problem over $\St(d,r)$ is to minimize the quadratic loss function on Stiefel manifold
\be\label{opt:consensus_thispaper}
\bad
\min   &\varphi^t(\X) :=\frac{1}{4} \sum_{i=1}^n\sum_{j=1}^n W^t_{ij} \normfro{x_i-x_j}^2 \\
&\st \quad x_i\in\M,\  \forall i\in[n],
\ead
\ee
where  the superscript  $t\geq 1$ is an integer used to denote the $t$-th power of a doubly stochastic matrix $W$. Note that $t$ is introduced   to provide flexibility for   algorithm design and analysis, and computing $W^t_{ij}$ corresponds to performing $t$ steps of communication on the tangent space.  
 The Riemannian gradient   method DRCS proposed in \cite{chen2020consensus} is given by for any $i\in[n]$,
  \be\label{consensus_rga}\bad
  x_{i,k+1} 
   = \Retr_{x_{i,k}}( \alpha\p_{\T_{x_{i}}\M} (\sum_{j=1}^n W^t_{ij} x_{j,k}) )  .\ead \ee
 

 DRCS   {converges} almost surely to consensus when $r\leq \frac{2}{3}d-1$ with random initialization \cite{markdahl2020high}. However, to study the decentralized optimization algorithm to solve \eqref{opt_problem}, the local Q-linear convergence of DRCS is more important for   decentralized optimization. Due to the nonconvexity of $\M$, the Q-linear rate of DRCS holds in a local region defined as follows
\begin{align}
 \N: &= \N_{1}\cap \N_{2}, \label{contraction_region} \\
	\N_{1} :& = \{ \X: \normfro{\X-\bar\X}^2 \leq n\delta_{1}^2 \}, \label{contraction_region_1} \\
	 \N_{2} : &= \{ \X: \normfroinf{\X-\bar\X} \leq \delta_{2} \}, \label{contraction_region_2}
\end{align}
    where $\delta_{1},\delta_{2}$
satisfy
\be\label{delta_1_and_delta_2} \delta_{1} \leq \frac{1 }{ 5\sqrt{r}}\delta_{2} \quad \textit{and} \quad    \delta_{2} \leq  \frac{1}{6}.  \ee


The following convergence result of DRCS can be found in  \cite[Theorem 2]{chen2020consensus}. The formal statement is provided in \cref{thm:linear_rate_consensus} in Appendix. 

\begin{fact}(\textbf{Informal})\label{thm:linear_rate_consensus_informal}
   Under \cref{assump:doubly-stochastic}, for some $\bar\alpha\in(0,1]$, if $\alpha \leq \bar\alpha$
 and $t\geq \lceil\log_{\sigma_2}(\frac{1}{2\sqrt{n}})\rceil$, 
	the sequence $\{\X_k\}$ of \eqref{consensus_rga} achieves consensus linearly if the initialization satisfies 
	$\X_0\in \N$ defined by \eqref{delta_1_and_delta_2}. That is, there exists 	  $\rho_t \in (0,1)$ such that $\X_k\in\N$ for all $k\geq 0$ and 
	\begin{align}\label{linear rate of consensus-0}
	\bad
	\normfro{\X_{k+1} - \bar{\X}_{k+1}}
	 \leq\rho_t\normfro{\X_{k} - \bar{\X}_{k}}.
	\ead
	\end{align}

\end{fact}


\section{Decentralized Riemannian gradient descent}\label{sec:distributed gradient method}


  The results of consensus problem on Stiefel manifold lead us to combine the ideas of decentralized gradient method in Euclidean space with the Stiefel manifold optimization. 
In this section, we study a distributed Riemannian stochastic gradient method for solving problem \eqref{opt_problem}, which is described in Algorithm \ref{alg:DRPG}. The algorithm is an   extension of the decentralized subgradient descent  \cite{nedic2010constrained}.

Since we need to  achieve consensus, the initial point $\X_0$ should be in the consensus region $\N$. One can simply initialize all agents from the same point. 
The step 5 in \cref{alg:DRPG} is first to perform a consensus step  and then to update local variable using Riemannian stochastic gradient direction $v_{i,k}$. The consensus step and computation of Riemannian gradient can be done in parallel\footnote{One could also exchange the order of gradient step and communication step, i.e., $x_{i,k+\half}=\Retr_{x_{i,k}}(-\beta_k v_{i,k}),$ $x_{i,k+1}=\Retr_{x_{i,k+\half}}( \alpha\p_{\T_{x_{i,k+\half}}\M} (\sum_{j=1}^n W^t_{ij} x_{j,k+\half}))$. Our analysis can also apply to this kind of updates if $\X_0\in \rho_t\N$, where $\rho_t\N$ denotes  region $\N$ with shrunk radius $\rho_t\delta_1, \rho_t\delta_2$. For  the Euclidean counterparts, when the graph is complete associated with equal weight matrix, the above updates are the same as centralized gradient step. However,  they are different on Stiefel manifold. }. The consensus stepsize $\alpha$ satisfies $\alpha\leq \bar\alpha$, which is the same as the consensus algorithm. The constant $\bar\alpha$ is given in \cref{thm:linear_rate_consensus} in Appendix. Moreover,   $\alpha=1$ works in practice for any $W$ satisfying \cref{assump:doubly-stochastic}.   
If $x_1=\ldots=x_n = z$, we denote
\[ f(z) := \frac{1}{n}\sum_{i=1}^{n} f_i (z).\]
  
Moreover, we need the following   assumptions on the stochastic Riemannian gradient $v_{i,k}$ and the stepsize $\beta_k$.

\begin{algorithm}[h]
	\caption{Decentralized Riemannian Stochastic Gradient Descent (DRSGD) for Solving \eqref{opt_problem}}\label{alg:DRPG}
	\begin{algorithmic}[1]
	 \State {\bfseries Input:}  initial point $\X_0 \in \N$, an integer $t\geq  \log_{\sigma_2}(\frac{1}{2\sqrt{n}}) $, $0< \alpha\leq \bar\alpha$, where $\bar\alpha$ is given in \cref{thm:linear_rate_consensus_informal}. 
		\For{$k=0,\ldots$\Comment{For each node $i\in[n]$, in parallel}}   
		\State Choose diminishing stepsize $\beta_k=\mathcal{O}(1/\sqrt{k})$ 
		\State Compute stochastic Riemannian gradient $v_{i,k}$ satisfying $\E v_{i,k} = \grad f_i(x_{i,k})$
		\State Update  
		$x_{i,k+1} = \Retr_{x_{i,k}}( \alpha\p_{\T_{x_{i,k}}\M} (\sum_{j=1}^n W^t_{ij} x_{j,k}) - \beta_{k} v_{i,k} )$
		\EndFor
	\end{algorithmic}
\end{algorithm}


\begin{assumption}\label{assum:stochastic-grad}
	\begin{enumerate}
		\item  The stochastic gradient $v_{i,k}$ is unbiased, i.e., $\E v_{i,k} = \grad f_i(x_{i,k})$ for all $i\in[n]$ and $v_{i,k}$ is independent of $v_{j,k}$ for any $i\neq j$. Moreover, the variance is bounded: $\E \normfro{v_{i,k} - \grad f_i(x_{i,k})}^2\leq \Xi^2$.
		\item We assume the uniform upper bound of $\normfro{v_{i}}$ is $D$, i.e.,  $\max_{x\in\St(d,r)}\normfro{v_i} \leq D$ for each $i\in[n]$.
	\end{enumerate}
\end{assumption}

The Lipschitz smoothness of $f_i(x)$ in \cref{assump:lips in Rn} and unbiased estimation   are quite standard in the literature. And \cref{lem:lipschitz} suggests that $\grad f_i$ is $L_G$-Lipschitz continuous. For the boundedness of $\normfro{v_i}$, it is a weak assumption since the Stiefel manifold is compact.
One common example is  the finite-sum form: $f_i=\frac{1}{m_i}\sum_{j=1}^{m_i} f_{ij},$ where $f_{ij}$ is smooth. Then the stochastic gradient $v_{i,k}$ is uniformly sampled from $\grad f_{ij}(x_{i,k}), j\in[m_i]$. We emphasize that the uniform boundedness of gradient is not needed for problems in Euclidean space, but Lipschitz continuity is necessary \cite{hong2020divergence}. The step 5 
can be seen as applying Riemannian  gradient method to  solve the following problem 
$$ \min_{\X\in\M^n} \beta_k f(\X) + \alpha \varphi^t(\X). $$
Similar as the analysis of DGD in Euclidean space, we need to ensure that $\normfro{\X_k-\bar{\X}_k}\rightarrow 0$. Hence, the effect of $f$ should be diminishing.   
The following assumption on the stepsize   is also needed to get an $\epsilon-$ solution. 

\begin{assumption}[Diminishing stepsize]\label{assump:stepsize_beta}
	The stepsize $\beta_k>0$   is non-increasing and 
	\[ \sum_{k=0}^\infty \beta_k = \infty ,\quad  \lim_{k\rightarrow\infty} \beta_k = 0, \quad \lim_{k\rightarrow \infty}\frac{\beta_{k+1}}{\beta_k}=1.\]
\end{assumption}

The assumption $\lim_{k\rightarrow \infty}\frac{\beta_{k+1}}{\beta_k}=1$ is additionally required to show the bound $\frac{1}{n}\normfro{\X_k-\bar\X_k}^2=\mathcal{O}(\frac{\beta_k^2D^2}{(1-\rho_t)^2})$, see   \cref{coro:rate_of_consensus} in  Appendix.

To proceed, we first need to guarantee that
$ \X_k \in\N,$
where $\N$ is the  consensus contraction region defined in \eqref{contraction_region}.  {Therefore, uniform bound $D$ and the multi-step consensus requirement $t\geq \lceil\log_{\sigma_2}(\frac{1}{2\sqrt{n}})\rceil$ are necessary in our convergence analysis. 
With appropriate stepsizes $\alpha$ and $\beta_k$, we get the following lemma using the consensus results in \cref{thm:linear_rate_consensus_informal}. We provide the proof in Appendix.

\begin{lemma}\label{lem:convergence_of_deviation from mean}
	Under \cref{assump:lips in Rn,assump:doubly-stochastic,assum:stochastic-grad,assump:stepsize_beta},
 let the stepsize  $\alpha$ satisfy $0<\alpha\leq\bar\alpha$,   $\beta_k$ satisfy $0\leq\beta_k\leq \min\{ \frac {1-\rho_t}{ D} \delta_1, \frac { \alpha\delta_1  }{5D}\}, \forall k\geq 0$, and $t\geq \lceil\log_{\sigma_2}(\frac{1}{2\sqrt{n}})\rceil$.   If $\X_0\in\N$, 	  it follows that $\X_k\in\N$ for all $k\geq 0$ generated by \cref{alg:DRPG}  and 
	\begin{align*} 
  \normfro{\X_{k+1} - \bar{\X}_{k+1}} &\le  \rho_t^{k+1} \normfro{\X_0 - \bar\X_0}  +    \sqrt{n} D \sum_{l=0}^k \rho_t^{k-l}\beta_l.
	\end{align*}
\end{lemma}

We have $\beta_k=\mathcal{O}(\frac{1-\rho_t}{D})$ when $\alpha = \mathcal{O}(1)$. Note that   $t\geq \lceil\log_{\sigma_2}(\frac{1}{2\sqrt{n}})\rceil$ implies $\rho_t=\mathcal{O}(1)$; see \cref{sec:append_consensus}. When $\beta_k$ is constant, \cref{lem:convergence_of_deviation from mean}   suggests that $\X_k$ converges linearly to   an $\mathcal{O}(\beta_k)$-neighborhood of $\bar\X_k$.

We   present the convergence of Algorithm \ref{alg:DRPG}. { The proof is based on the new Lipschitz inequalities for the Riemannian gradient in \cref{lem:lipschitz} and the properties of retraction in \cref{lem:nonexpansive_bound_retraction}. We provide it in Appendix. }


\begin{theorem}\label{thm:convergence of alg 2}
	Under \cref{assump:lips in Rn,assump:doubly-stochastic,assum:stochastic-grad,assump:stepsize_beta}, suppose $\X_k\in\N$,   $t\geq \lceil\log_{\sigma_2}(\frac{1}{2\sqrt{n}})\rceil$,  $0<\alpha\leq\bar\alpha$.   If {\small\be\label{stepsize_betak}\beta_k =\frac{1}{\sqrt{k+1}} \cdot \min\{ \frac{1}{5L_g}, \frac{\alpha \delta_1  }{5D}, \frac {1-\rho_t}{ D} \delta_1 \},\ee  }
	it follows that
	{\small\begin{align}\label{ineq:convergence DRSGA}
	&\quad \min_{k\leq K} \E\normfro{\grad f(\bar x_k)}^2  
	 \leq  \frac{ 4(f(\bar x_0) - f^*) +  \frac{6L_g\Xi^2}{ n} \sum_{k=0}^{K}\beta_k^2}{\sum_{k=0}^{K}\beta_k} \\
	 &\quad +  \frac{  (2CD^2 L_G^2 + 4\cT_1 D^4)\sum_{k=0}^{K}\beta_k^3    + 4\cT_2 L_gD^4\sum_{k=0}^{K}\beta_k^4}{ 	\sum_{k=0}^{K}\beta_k}, \notag
	\end{align}}
	where   $C=\mathcal{O}(\frac{1}{(1-\rho_t)^2})$ is  given in \cref{coro:rate_of_consensus} in  Appendix. And   $\cT_1=2(4\sqrt{r} + 6 \alpha)^2 C^2+ 8 M^2$ and $\cT_2 = 201\alpha^2C^2 +  9M^2 . $
	Therefore, we have
	{\small\begin{align*}
	 	&\min_{k\leq K}\E\normfro{\grad f(\bar x_k)}^2
= \mathcal{O}\left(\frac{  f(\bar{ x}_0) - f^*  }{\tilde \beta\sqrt{K+1}} +\frac{ \Xi^2\ln (K+1)}{n\sqrt{K+1}}\right) \notag\\
	&\quad + \mathcal{O}\left( \frac{\max\{D^2,L_G^2\}\cdot( C + \cT_1 + \cT_2) }{ \sqrt{K+1}}\right) ,
	\end{align*}}
	where $\tilde \beta = \min\{1/L_g, (1-\rho_t)/D\}$. 
\end{theorem}

\cref{thm:convergence of alg 2} together with \cref{coro:rate_of_consensus} implies that the iteration complexity of obtaining an $\epsilon-$stationary point defined in \cref{def:stationary} is $\mathcal{O}(1/\epsilon^2)$ in expectation. 
The communication round per iteration is   $t\geq \lceil\log_{\sigma_2}(\frac{1}{2\sqrt{n}})\rceil$  since we need to ensure $\X_k\in\N$. For sparse network, $t= \mathcal{O}(n^2 \log n)$ \cite{chen2020consensus}.   \\ 
 Following \cite{lian2017can},   if we use the constant stepsize $\beta_k=\frac{1}{2L_G+ \sqrt{(K+1)/n}}$ where $K$ is sufficiently large, we can obtain the following result  
 \begin{align*}&\quad \min_{k=0,\ldots,K} \E\normfro{\grad f(\bar x_k)}^2\\
 &\leq   \frac{ 8L_G(f(\bar x_0) - f^*)}{K+1} +\frac{ 8(f(\bar x_0) - f^*+ \frac{3L_G}{2})\Xi}{\sqrt{n(K+1)}}.\end{align*}
 More details are provided in \cref{coro:DRSGD constant stepsize} in Appendix. Therefore, if $K$ is sufficiently large, the convergence rate is $\mathcal{O}(1/\sqrt{nK}).$ To obtain an  $\epsilon-$stationary point, the computational complexity of single node is $\mathcal{O}(\frac{1}{n\epsilon^2}).$ However, the communication round  $t\geq \lceil\log_{\sigma_2}(\frac{1}{2\sqrt{n}})\rceil$ is too large. In practice, we find $t=1$ performs almost the same as $t=\infty$, which is shown in \cref{sec:numerical}.   This may be because that when the stepsize is very small, DRSGD will not deviate from the consensus algorithm DRCS too much. We leave the further discussion as future work. 

	\section{Gradient tracking on Stiefel manifold}\label{sec:gradient-tracking}
  In this section we study the decentralized gradient tracking method, which is based on the DIGing algorithm\cite{qu2017harnessing,nedic2017achieving} for solving Euclidean problems. With an auxiliary  gradient tracking sequence to estimate the full gradient, the constant stepsize can be used and faster convergence rate can be shown for the Euclidean algorithms \cite{nedic2017achieving,shi2015extra}. 
We describe our algorithm in Algorithm \ref{alg:DRG_GT}, which is named as Decentralized Riemannian  Gradient Tracking Algorithm (DRGTA).  
\begin{algorithm}[ht]
	\caption{Decentralized Riemannian  Gradient Tracking  over Stiefel manifold (DRGTA) for Solving \eqref{opt_problem}}\label{alg:DRG_GT}
	\begin{algorithmic}[1]
		\State{ Input: initial point $\X_0 \in \N$,  an integer $t\geq  \log_{\sigma_2}(\frac{1}{2\sqrt{n}}) $, $0< \alpha\leq \bar\alpha$  and stepsize $\beta$ according to \eqref{stepsize_beta}. } 
		\State Let  $y_{i,0} = \grad f_i(x_{i,0})$ on each node $i\in[n]$. 
		\For{$k=0,\ldots$\Comment{For each node $i\in[n]$, in parallel}} 
		\State{Projection onto tangent space: $v_{i,k} = \p_{\T_{x_{i,k}}\M}y_{i,k} $.}
		\State{Update  $ x_{i,k+1} = \Retr_{x_{i,k}}(\alpha\p_{\T_{x_{i,k}}\M} (\sum_{j=1}^n W^t_{ij} x_{j,k}) - \beta v_{i,k}).$ }
		\State{Riemannian gradient tracking:  \[y_{i,k+1} = \sum_{j=1}^n W^t_{ij}  y_{j,k} + \grad f_i(x_{i,k+1}) - \grad  f_i(x_{i,k}).   \] }
		\EndFor
	\end{algorithmic}
\end{algorithm}

In \cref{alg:DRG_GT}, the step 4 is to project the direction $y_{i,k}$ onto the tangent space $\T_{x_{i,k}}\M$, which follows a retraction update.    The sequence $\{ y_{i,k}\}$ is to approximate the Riemannian gradient $\grad f_i(x_{i,k})$. More specifically, the sequence $\{ y_k\}$ tracks the average Riemannian gradient $\frac{1}{n}\sum_{i=1}^n\grad f_i(x_{i,k})$. Although, it is not mathematically sound to do addition operation between different tangent space in differential geometry,  we can view  $\grad f_i(x_{i,k})$ as the projected Euclidean gradient.   {Note that $y_{i,k}$ is not necessarily on the tangent space $\T_{x_{i,k}}\M.$   Therefore, it is important to define $v_{i,k} = \p_{\T_{x_{i,k}}\M} y_{i,k}  $  so that we can use the properties of retraction in \cref{lem:nonexpansive_bound_retraction}. {Such a projection onto tangent space step, followed by the retraction operation, distinguishes the algorithm from the Euclidean space gradient tracking algorithms.}  Multi-step consensus of gradient is also required in step 5 and step 6. The consensus stepsize $\alpha$ satisfies the same condition as that of \cref{alg:DRPG}. 

\subsection{Convergence of Riemannian gradient tracking}
We first briefly revisit the idea of gradient tracking (GT) algorithm  DIGing in Euclidean space. Note that if we consider the decentralized optimization problem \eqref{opt_problem} without the Stiefel manifold constraint,  then  \cref{alg:DRG_GT} is exactly the same as the DIGing. Since the Riemannian gradient $\grad f_i$ becomes simply the Euclidean gradient $\nabla f_i$ and   projection onto the tangent space and retraction  are not needed.
The main advantage of Euclidean gradient tracking algorithm is that one can use constant stepsize $\beta>0$, which is due to following observation:  for all $k\geq 0$, it follows that
\[\frac{1}{n}\sum_{i=1}^n y_{i,k} =\frac{1}{n}\sum_{i=1}^n \nabla f_i(x_{i,k}).  \]
That is, the average of sequence $y_{i,k}$ is the same as that of $\nabla f_i(x_{i,k})$. It can be shown that  if the following inexact gradient sequence, then it  converges to   a stationary point  \cite{nedic2017achieving}
\[  x_{i,k+1} = \sum_{i=1}^n W_{ij} x_{j,k}  - \beta\frac{1}{n}\sum_{i=1}^n \nabla f_i(x_{i,k}).  \]
 However, the average of gradient information is unavailable in the decentralized setting.  Therefore, GT uses $\frac{1}{n}\sum_{i=1}^n y_{i,k}$ to  approximate $\frac{1}{n}\sum_{i=1}^n \nabla f_i(x_{i,k})$.   Inspired by this,   $y_{i,k}$  is used to approximate the  Riemannian gradient, i.e., if 
\[ y_{i,k+1} = \sum_{j=1}^n W^t_{ij}  y_{j,k} + \grad f_i(x_{i,k+1}) - \grad  f_i(x_{i,k}) ,   \]
then it follows that
\[ \frac{1}{n}\sum_{i=1}^n y_{i,k} = \frac{1}{n}\sum_{i=1}^n \grad f_i(x_{i,k})\quad \ie \quad \hat y_k = \hat g_k.  \]
Therefore, $\{\Y_k\}$ tracks the average of Riemannian gradient, and if $\normfro{\hat g_k}\rightarrow 0$ and the sequence $\{\X_k\}$ achieves consensus, then $\X_k$ also converges to the critical point. 
 This is because	\begin{align*} 
 &\quad\normfro{\grad f(\bar x_k)}^2 \leq 2\normfro{\hat g_k}^2 + 2\normfro{\grad f(\bar x_k)-\hat g_k}^2\\
 &\stackrel{\eqref{ineq:lips_riemanniangrad}}{\leq} 2\normfro{\hat g_k}^2 + \frac{2L_G^2}{n}\normfro{\X_k-\bar\X_k}^2. \end{align*} 
 {To achieve consensus, we still need multi-step consensus in DRGTA as DRSGD. The multi-step consensus also helps us to show the uniform boundedness of $y_{i,k}$ and $v_{i,k}$, $i\in[n]$ for all $k\geq 0,$ which is important to guarantee $\X_k\in\N$. } We   get that the sequence stays in consensus region $\N$ in \cref{lem:uniform bound y}. We provide the proof  in Appendix. 


\begin{lemma}[Uniform bound of $y_i$ and stay in $\N$]\label{lem:uniform bound y}
	Under \cref{assump:lips in Rn,assump:doubly-stochastic},	let $\X_0\in\N$, $t\geq \log_{\sigma_2}(\frac{1}{2\sqrt{n}})$, $\alpha$ satisfy $0<\alpha\leq\bar\alpha$,   $\beta$ satisfy $0\leq\beta\leq\bar\beta:= \min\{ \frac {1-\rho_t}{ L_G+2D} \delta_1, \frac { \alpha\delta_1  }{5(L_G+2D)}\}$, then $\normfro{y_{i,k}}\leq L_G+2D$ for all $i\in[n]$ and $\X_k\in\N$ for all $k\geq 0$. Moreover, we have \be\label{ienq:bound_x_gt}\frac{1}{n}\normfro{\X_k-\bar\X_k}^2\leq  C_1(L_G+2D)^2\beta^2, k \geq 0\ee for some $C_1=\mathcal{O}(\frac{1}{(1-\rho_t)^2})$ and $C_1$ is independent of $L_G, D$.  
\end{lemma}

We   present the $\mathcal{O}(1/\epsilon)$ iteration complexity to obtain the $\epsilon-$stationary point of \eqref{opt_problem} as follows. The proof of DIGing  can be unified by the primal-dual framework  \cite{alghunaim2019decentralized}. However, DRGTA cannot be rewritten in the primal-dual form.   {The proof is mainly established with the help of \cref{lem:lipschitz} and the properties of IAM.  }  We provide it in Appendix.
\begin{theorem}\label{thm:theorem-convergence-grad-tracking}
		Under \cref{assump:lips in Rn,assump:doubly-stochastic},  let    $\X_0\in\N$, $t\geq \lceil\log_{\sigma_2}(\frac{1}{2\sqrt{n}})\rceil$,  $0<\alpha\leq\bar\alpha$,  and 
	\be\label{stepsize_beta} 0< \beta\leq \min\{  \bar\beta,\frac{1}{8L_G},\frac{1}{4L_G(2 \G_3  + (8 \tilde{C}_0   +  \frac{1}{2} \tilde{C}_2 ) \alpha \delta_1)}\}, \ee
	where $\bar\beta$ is given in \cref{lem:uniform bound y}.  
	Then it follows that for the sequences generated by \cref{alg:DRG_GT} 
	 \be\label{convergence_Y}
\min_{k=0,\ldots,K} \frac{1}{n}\normfro{\Y_k}^2
	\leq  \frac{8(f(\bar x_0) -  f^*  +   \tilde{C}_4 +    \G_4 L_G )}{\beta\cdot K},\ee
	\begin{equation}\label{convergence_X}
	\bad
	&\quad\min_{k\leq K}\frac{1}{n} \normfro{\X_k - \bar\X_k}^2\\
	&\leq  \frac{8\beta(f(\bar x_0) -  f^*  +   \tilde{C}_4+     \G_4 L_G ) \tilde C_0 + \tilde C_1 }{  K},
	\ead\end{equation}
	\be \label{convergence_grad} \bad
	&\quad \min_{k\leq K} \normfro{\grad f(\bar x_k)}^2 \\
	&\leq {\small\frac{(16+\alpha^2 \delta_1^2\tilde C_0)(f(\bar x_0) -  f^*  +   \tilde{C}_4+     \G_4 L_G ) + \tilde C_1 L_G }{\beta\cdot K},}\ead\ee
	where the constants above are given by
	
	\begin{align*}
	&\G_3 =\G_1\tilde{C}_0 +\G_0\tilde{C}_0  + \G_2, \\& \G_4 = \frac{ \G_0\tilde{C}_0  \delta_1^2\alpha^2  }{25}  + \tilde{C}_1(\G_1+4rC_1),\\
	&\tilde{C}_0 = \frac{2}{(1-\rho_t)^2}, \quad   \tilde{C}_1 = \frac{2}{1-\rho_t^2} \cdot\frac{1}{n}\normfro{\X_0 - \bar \X_0}^2, \\
	& \tilde{C}_2 = \frac{2}{(1-\sigma_2^t)^2}, \quad \tilde{C}_3 = \frac{2}{1-\sigma_2^{2t}} \cdot	\frac{1}{n}\normfro{\Y_0 - \hat\bG_0}^2,\\
	&\tilde{C}_4 = ( 8\alpha^2\tilde{C}_1\tilde{C}_2L_G^2 +\tilde{C}_3)\cdot\frac{\beta}{2}=\mathcal{O}(\frac{   L_G}{(1-\sigma_2^t)^2}).
	 \end{align*} 
\end{theorem}
The constants $\G_0 = \mathcal{O}(r^2C_1)$ , $\G_1 = \mathcal{O}(r^2C_1)$ and $\G_2 = \mathcal{O}(M)$ are given in \cref{lem:descent-grad-tracking} in the appendix. 
We have $\G_3 =\mathcal{O} (\frac{r^2C_1}{(1-\rho_t)^2} + M)$ and $\G_4 =\mathcal{O} (\frac{r^2 C_1\delta_1^2}{1-\rho_t^2} )$. Recall that $\beta\leq  \bar \beta$ is required to guarantee that the sequence $\{\X_k\}$ always stays in the consensus region $\N$. And note that $\rho_t$ is the linear rate of Riemannian consensus, which is greater than $\sigma_2^t.$ The stepsize $\beta$ follows  \[ \beta =\mathcal{O}(\min\{\frac{1-\rho_t}{L_G+2D}, \frac{(1-\rho_t)^2}{L_G}\cdot\frac{1}{r^2C_1 + M(1-\rho_t)^2}\}).\]
This matches the bound of DIGing \cite{qu2017harnessing,nedic2017achieving}. 
Then \cref{thm:theorem-convergence-grad-tracking} suggests that  the consensus error rate is $\mathcal{O}( \frac{1 }{(r^2C_1+M) L_G}\cdot\frac{f(\bar x_0 ) - f^*}{K} + \frac{\normfro{\X_0-\bar\X_0}^2}{n(1-\rho_t^2)K})$ and the convergence rate of  $\min\limits_{k=0,\ldots,K} \normfro{\grad f(\bar x_k)}^2 $ is given by $\mathcal{O}(\frac{(r^2C_1+M)(L_G+2D)(f(\bar x_0 ) - f^*)) }{ K(1-\rho_t)^2 }+ \frac{\normfro{\X_0-\bar\X_0}^2}{n(1-\rho_t)^4T}+\frac{r^2C_1 \delta_1^2 L_G}{ K (1-\rho_t)^6})$. Moreover,  if the initial points satisfy $x_{1,0}=x_{2,0}=\ldots=x_{n,0},$ we have $\tilde C_1 = \tilde C_3 = \tilde C_4 = 0.$



\section{Numerical experiment}\label{sec:numerical}

We solve the following decentralized eigenvector problem:
\be\label{prob:eigenvector}\min_{\X\in \M^n}  -\frac{1}{2n}\sum_{i=1}^n x_i^\top A_i^\top A_i x_i,\quad \st \quad x_1=\ldots=x_n,  \ee
where $A_i\in\R^{m_i\times d}, i\in[n]$ is the local data matrix in local agent and $m_i$ is the sample size. Denote the global data matrix by $A:= [A_1^\top \ A_2^\top \ldots A_n^\top ]^\top$. It is known that the global minimizer of \eqref{prob:eigenvector} is given by the first $r$ leading eigenvectors of $A^\top A =\sum_{i=1}^n A_i^\top A_i,$ denoted by $x^*$.   DRSGD and DRGTA are only proved to converge to the critical points, but we find they always converge to $x^*$ in our experiments. Denote the column space of a matrix $x$ by $[x]$.   {{To measure the quality of the solution}},  the distance between column space $[x]$ and $[y]$ can be defined via the canonical correlations between $x\in\R^{d\times r}$ and $y\in\R^{d\times r}$\cite{golub1995canonical}. One can define it by
\[ d_s(x,y):= \min_{Q\in\mathrm{O}(r)} \normfro{u Q - v },\]
where $\mathrm{O}(r)$ is the orthogonal group, $u$ and $v$ are the orthogonal basis of $[x]$ and $[y]$, respectively. In the sequel, we fix $\alpha=1$ and generate  the initial points uniformly randomly satisfying $x_{1,0}=\ldots=x_{n,0}\in\M$. If full batch gradient is used in \cref{alg:DRPG}, we call it DRDGD, otherwise one stochastic gradient is uniformly sampled without replacement  in DRSGD. In DRSGD, one epoach represents  the number of complete passes through the dataset, while one iteration is used in the deterministic algorithms. For DRSGD, we set the maximum epoch to $200$ and early stop it if $d_s(\bar{x}_k, x^*)\leq 10^{-5}$. 
For DRGTA and DRDGD, we set the maximum iteration number to $10^4$  and the termination condition is $d_s(\bar{x}_k, x^*)\leq 10^{-8}$ or $\normfro{\grad f(\bar x_k)}\leq 10^{-8}.$  We set  $\beta_k= \frac{\hat \beta}{\frac{1}{n}\sum_{i=1}^n m_i}$ for DRGTA and DRDGD where $\hat \beta$ will be specified later. For DRSGD, we set $\beta = \frac{\hat \beta}{\sqrt{200}}$. We  select the weight matrix $W$ to be the  Metroplis constant weight \cite{shi2015extra}. 
\subsection{Synthetic data}
We   report the convergence results of DRSGD, DRDGD and DRGTA with different $t$ and $\hat\beta$ on synthetic data. 
We fix $m_1=\ldots=m_n=1000$,  $d=100$ and $r=5$ and generate $m_1\times n$ i.i.d  samples   following standard multi-variate  Gaussian distribution to obtain $A$. Let $A=USV^\top$ be the truncated SVD. Given an eigengap $\Delta\in(0,1)$, we modify the singular values of $A$ to be a geometric sequence, i.e. $S_{i,i}=S_{0,0}\times \Delta^{i/2}, i\in[d].$ Typically, larger $\Delta$ results in more difficult problem.  
\begin{figure*}[ht]
	\begin{center}
		\minipage{0.33\textwidth}
		\subfigure[DRSGD]{
			{\includegraphics[width= \linewidth]{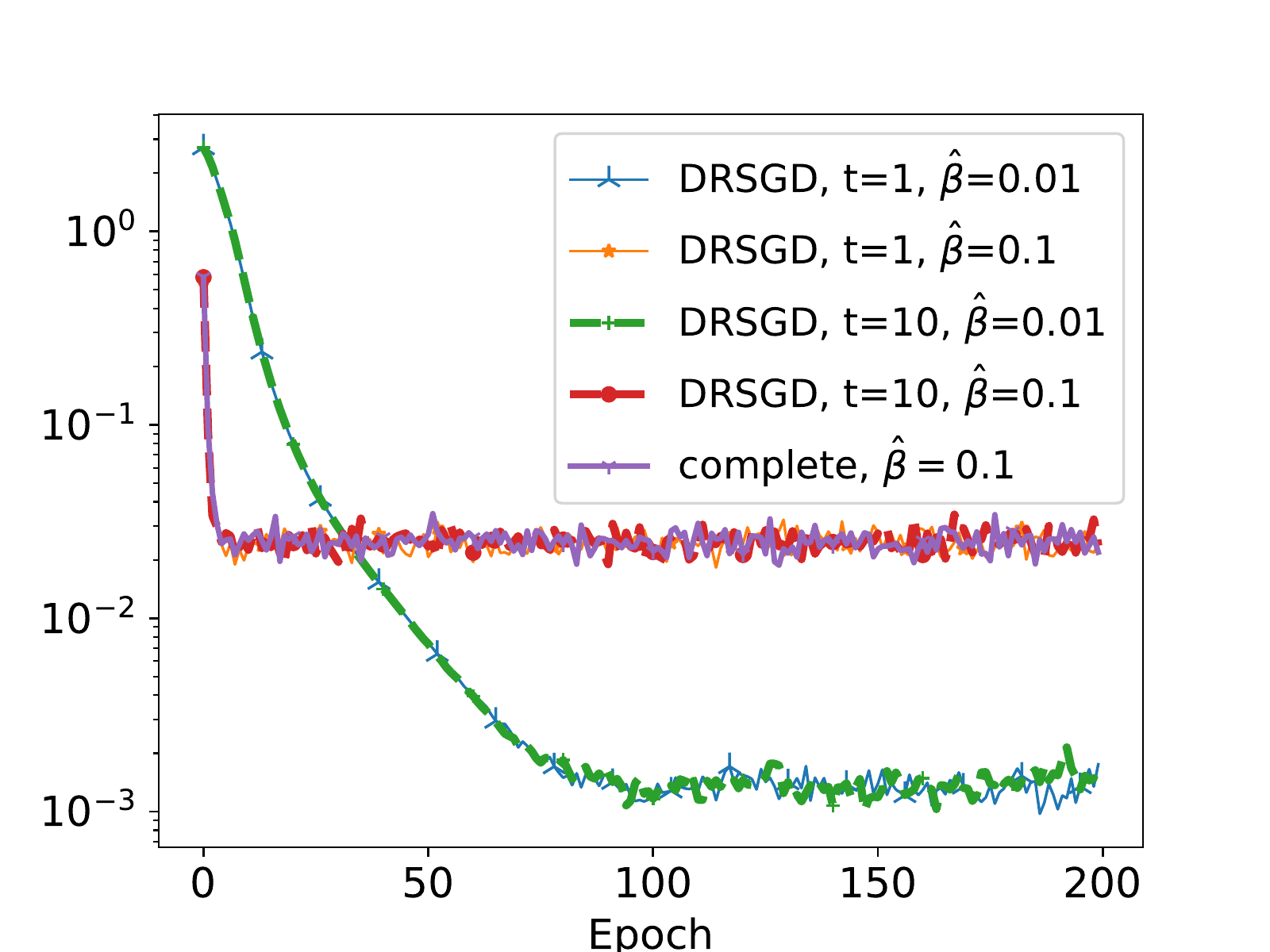}} }
		\endminipage
		\minipage{0.33\textwidth}
		\subfigure[DRDGD]{
			{\includegraphics[width=\linewidth]{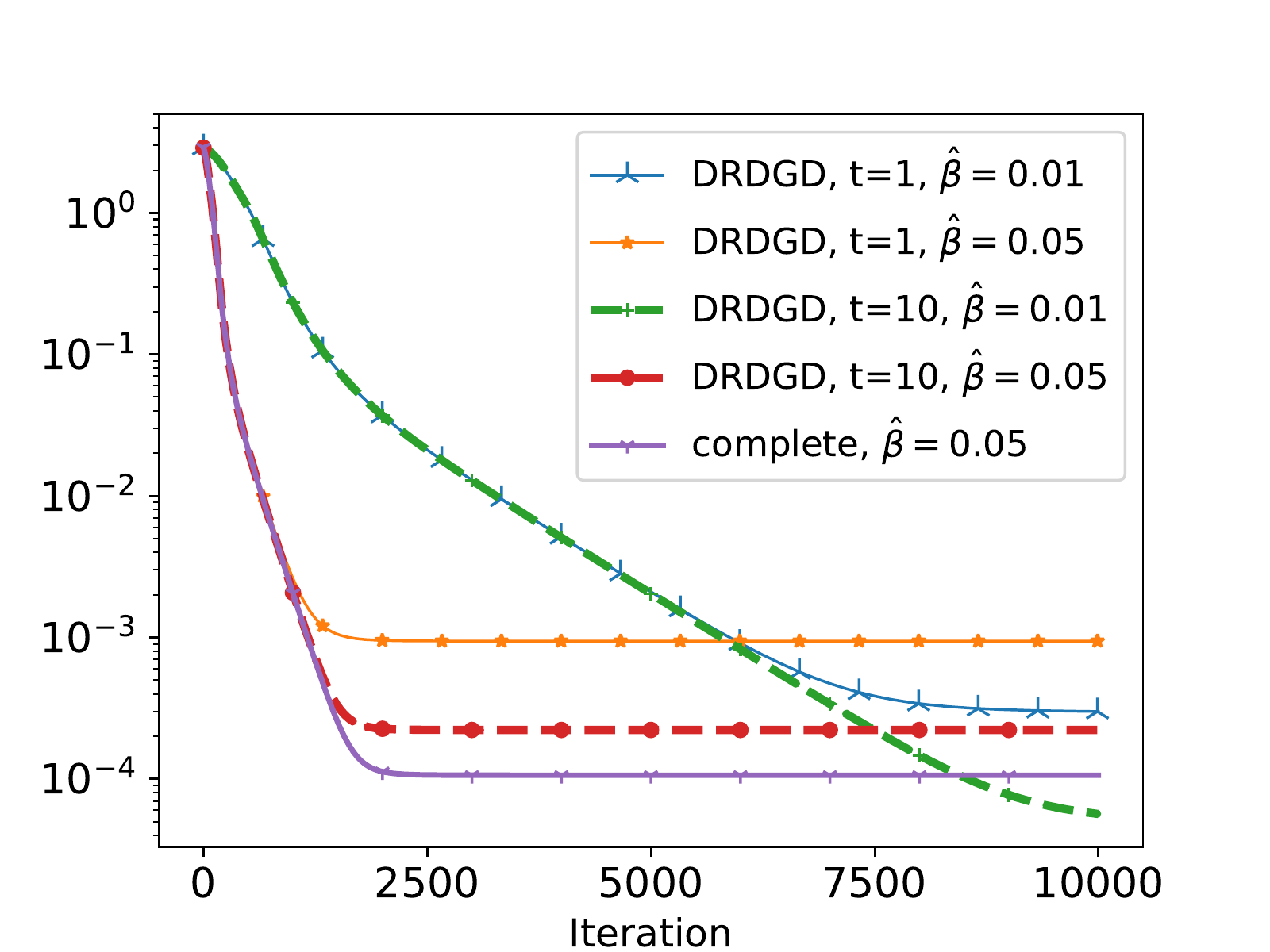}} }
		\endminipage
		\minipage{0.33\textwidth}
		\subfigure[DRGTA]{
			{\includegraphics[width= \linewidth]{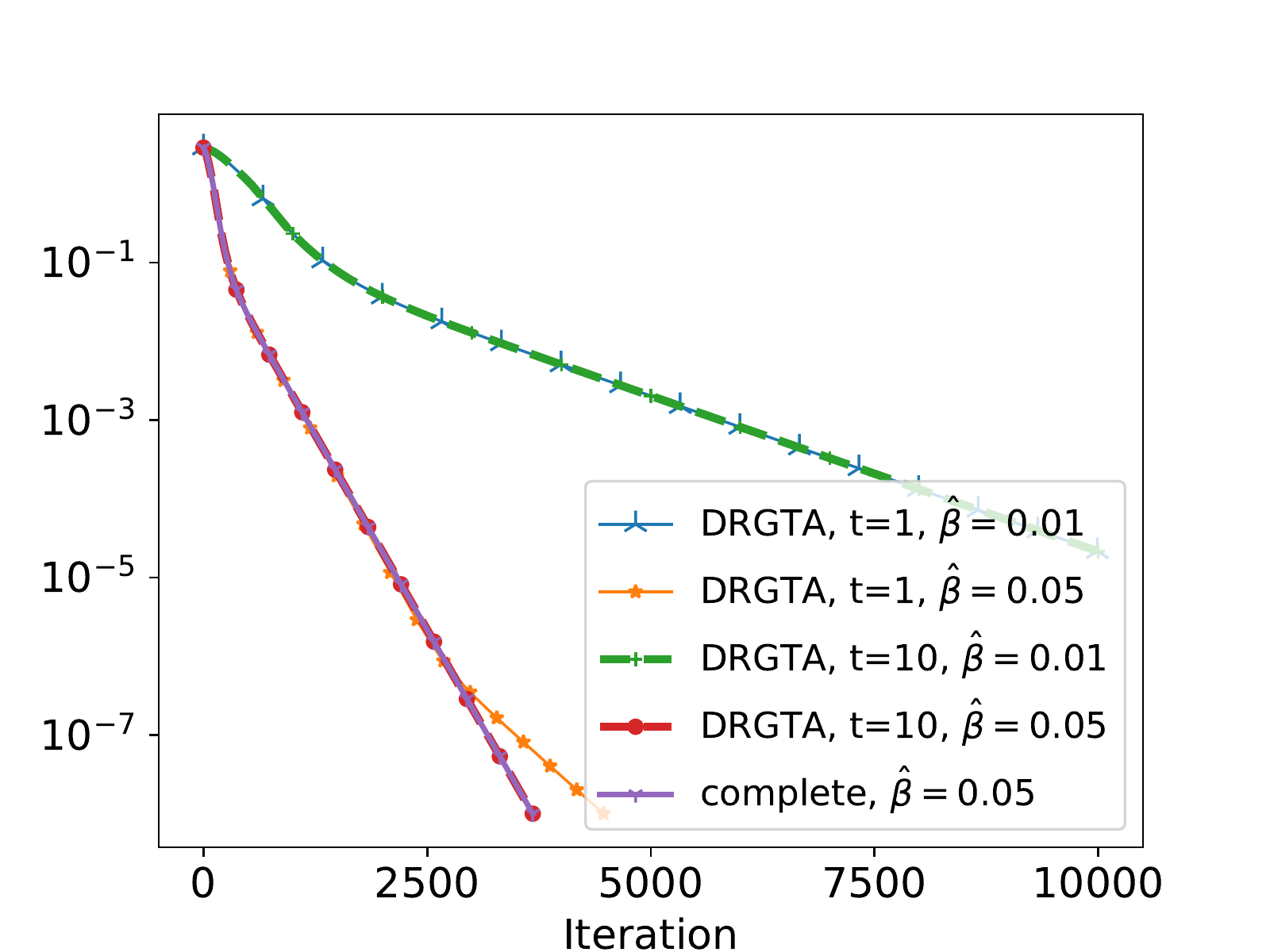}} }
		\endminipage
		\caption{Synthetic data, agents number $n=32,$ eigengap $\Delta=0.8$. y-axis: log-scale $d_s(\bar x_k, x^*)$. }\label{fig:linear_syn}
		\label{figure:synthetic-n-32}
	\end{center}
\end{figure*}
In \cref{figure:synthetic-n-32}, we show the results of DRSGD, DRDGD and DRGTA on the data with $n=32$ and $\Delta=0.8.$ The y-axis is the log-scale distance $d_s(\bar x_k, x^*).$ The first four lines in each testing case  are for the ring graph, and the last one is on a complete graph with equally weighted matrix, which aims to show the case of  $t\rightarrow \infty$. In \cref{figure:synthetic-n-32}(a), when fixing $\hat \beta$, it is shown that  that smaller $\hat \beta$ produces higher accuracy, which indicates the \cref{thm:convergence of alg 2}. We also see DRSGD performs almost the same with different $t\in\{1,10,\infty\}$. 
For the two deterministic algorithms DRDGD and DRGTA, we see that DRDGD can use larger $\hat\beta$ if more communication rounds $t$ is used in \cref{figure:synthetic-n-32}(b),(c). DRDGD cannot achieve exact convergence with the constant stepsize, while DRGTA successfully solves the problem using $t\in\{1, 10,\infty\}, \hat\beta=0.05.$  

Next, we report the numerical results on different networks and data size.

\cref{figure:synthetic-n-32-er} shows the results on the same data set as that of  \cref{figure:synthetic-n-32}. However, the network is an Erd\"{o}s-R\'{e}nyi model $\mathsf{ER}(n, p)$, which means the probability of each edge is included in the graph with probability $p$.   The Metropolis constant matrix is associated with the graph.  Since the $\mathsf{ER}(32, 0.3)$ is more well-connected than the ring graph, we see that the results for different $t\in\{1,10,\infty\}$ are almost the same except for DRDGD with $\hat\beta=0.05$. Moreover, the solutions accuracy and convergence rate of DRDGD and DRGTA are better than those shown in \cref{figure:synthetic-n-32}.  

\begin{figure*}[ht]
	\begin{center}
		\minipage{0.33\textwidth}
		\subfigure[DRSGD]{
			{\includegraphics[width= \linewidth]{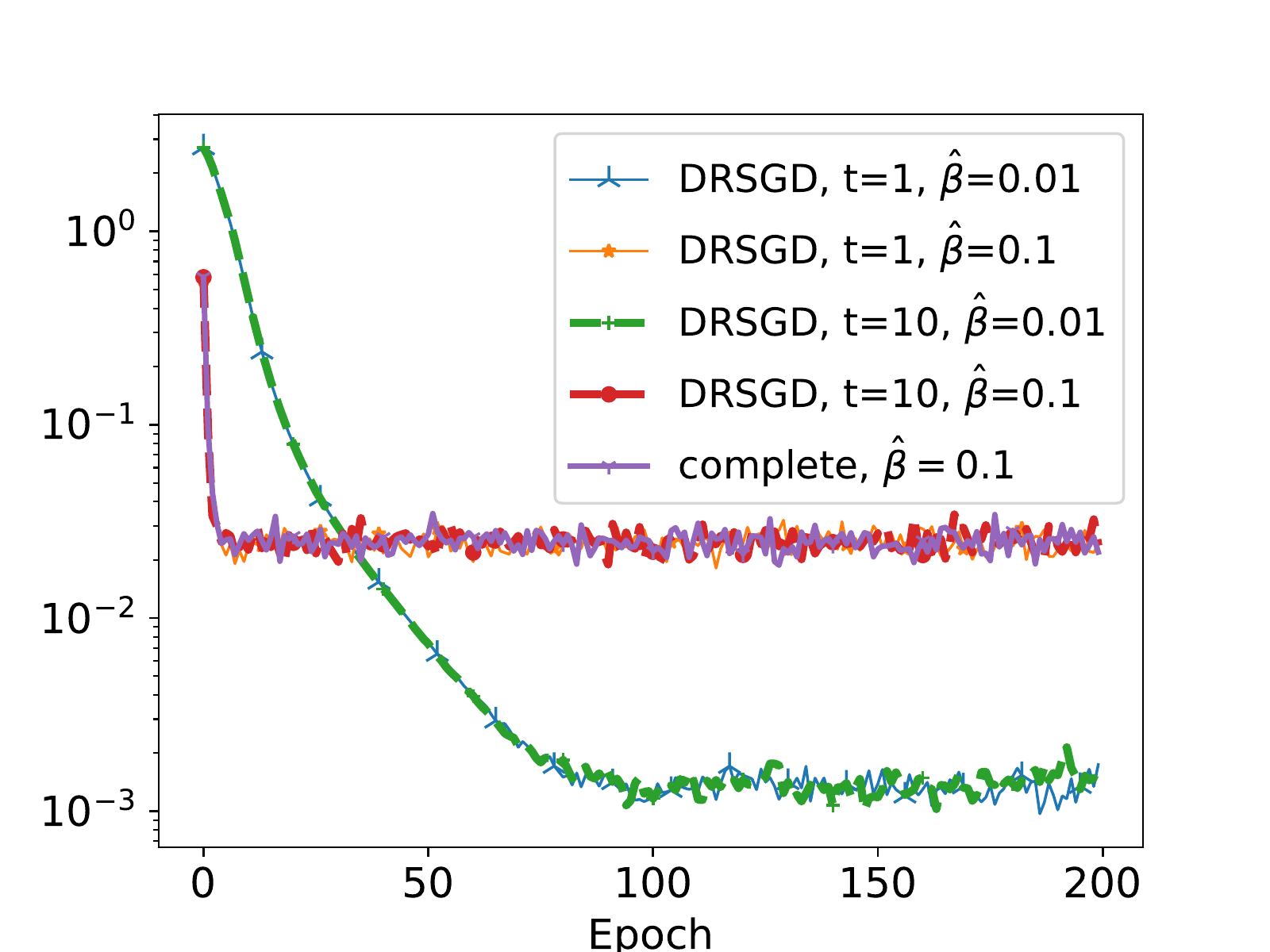}} }
		\endminipage
		\minipage{0.33\textwidth}
		\subfigure[DRDGD]{
			{\includegraphics[width=\linewidth]{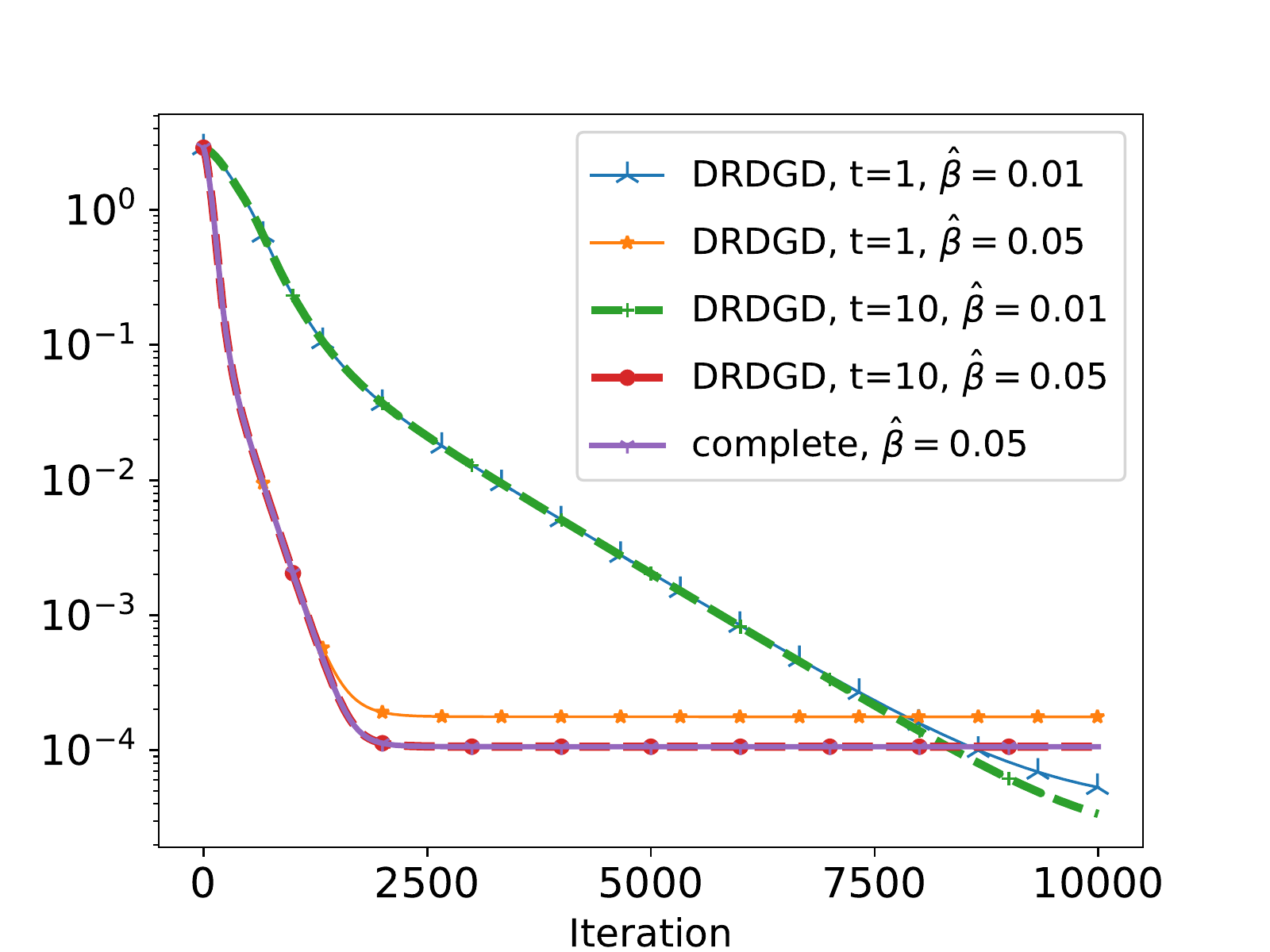}} }
		\endminipage
		\minipage{0.33\textwidth}
		\subfigure[DRGTA]{
			{\includegraphics[width= \linewidth]{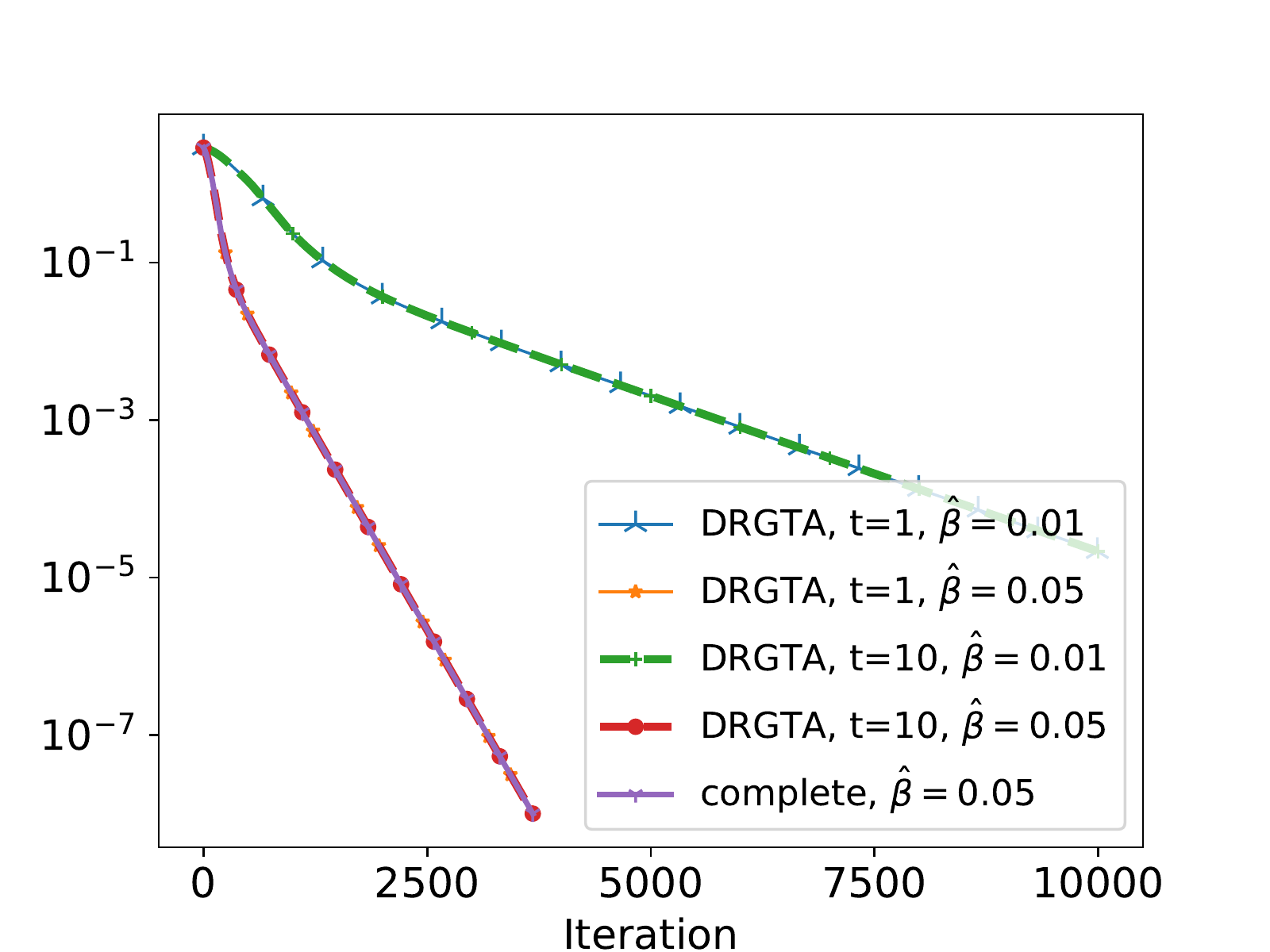}} }
		\endminipage
		\caption{Synthetic data, agents number $n=32,$ eigengap $\Delta=0.8$, Graph: $\mathsf{ER}(32, 0.3).$  y-axis: log-scale $d_s(\bar x_k, x^*)$. } 
		\label{figure:synthetic-n-32-er}
	\end{center}
\end{figure*}

In \cref{figure:synthetic-n-32-different start}, we show the results when the initial point does not satisfy $\X_0\in\N$. Specifically, we  randomly generate $x_{1,0}, \ldots, x_{n,0}$ on $\M,$ and the other settings are the same as \cref{figure:synthetic-n-32}. Surprisingly, we find that the proposed algorithms still converge. As suggested by \cite{markdahl2020high,chen2020consensus}, the consensus algorithm can achieve global consensus with random initialization when $r\leq \frac{2}{3}d-1$. The iteration in DRSGD and DRGTA is a perturbation of the consensus iteration.  It will be interesting to  study it further.  

\begin{figure*}[ht]
	\begin{center}
		\minipage{0.33\textwidth}
		\subfigure[DRSGD]{
			{\includegraphics[width= \linewidth]{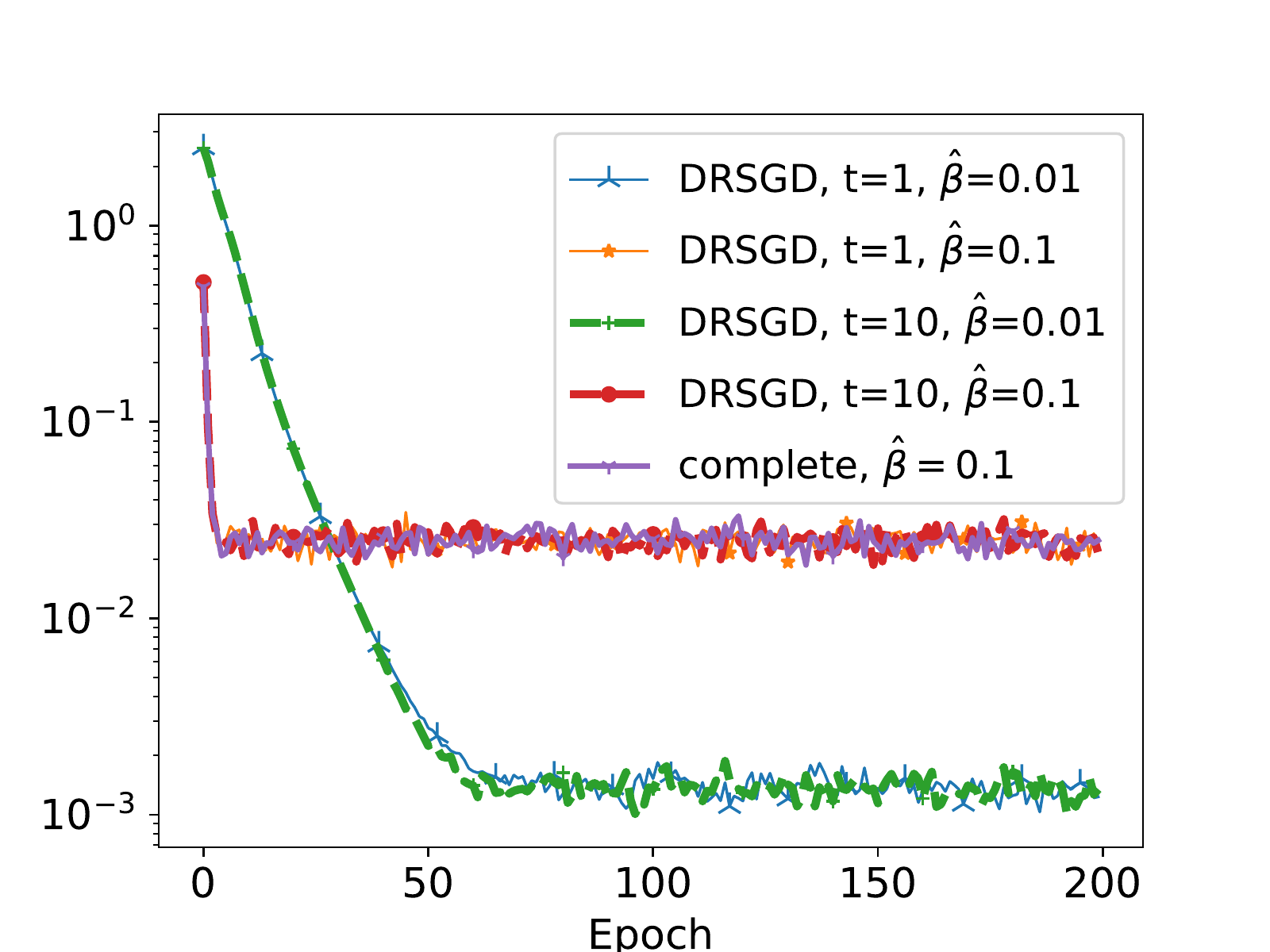}} }
		\endminipage
		\minipage{0.33\textwidth}
		\subfigure[DRDGD]{
			{\includegraphics[width=\linewidth]{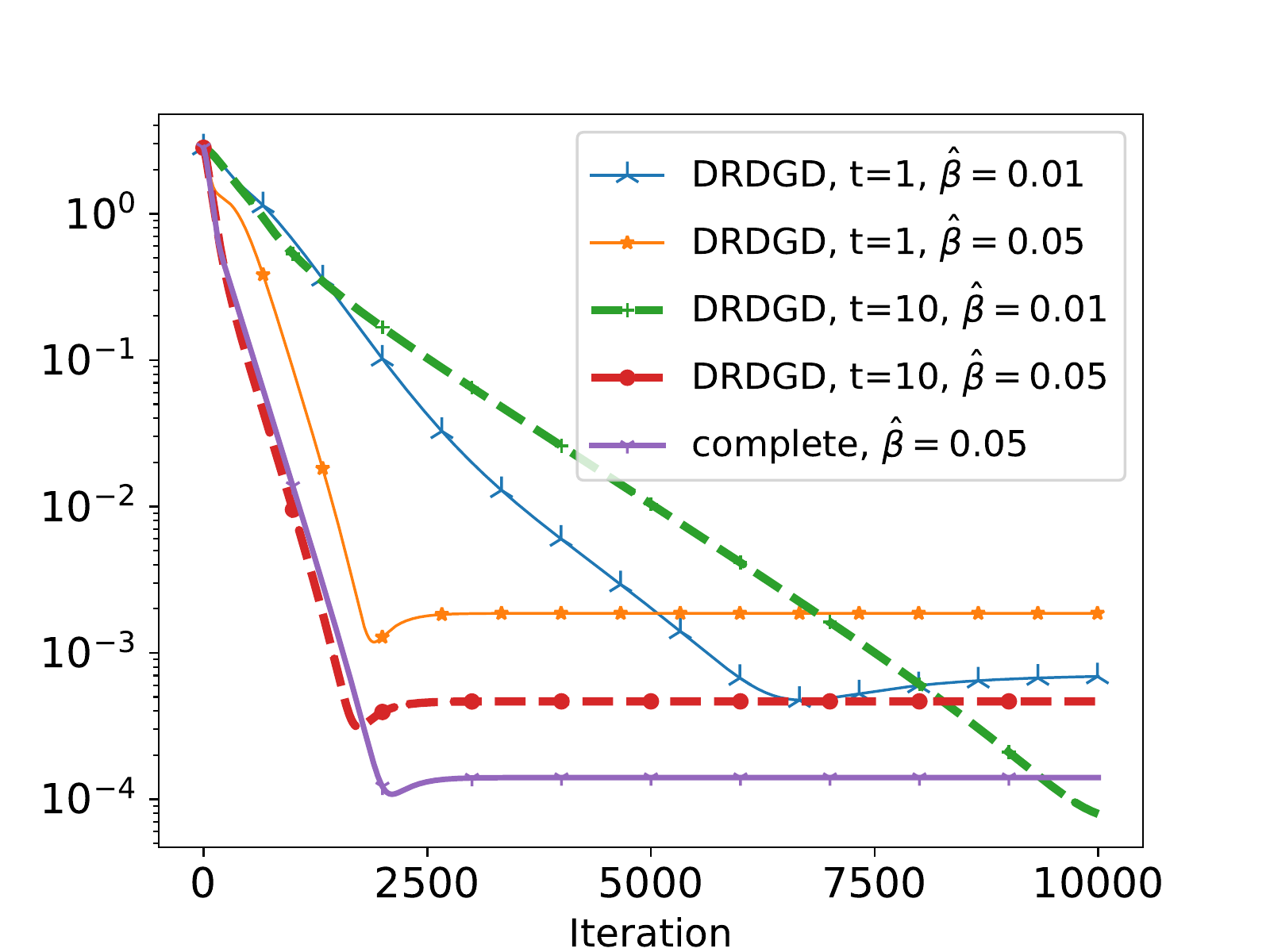}} }
		\endminipage
		\minipage{0.33\textwidth}
		\subfigure[DRGTA]{
			{\includegraphics[width= \linewidth]{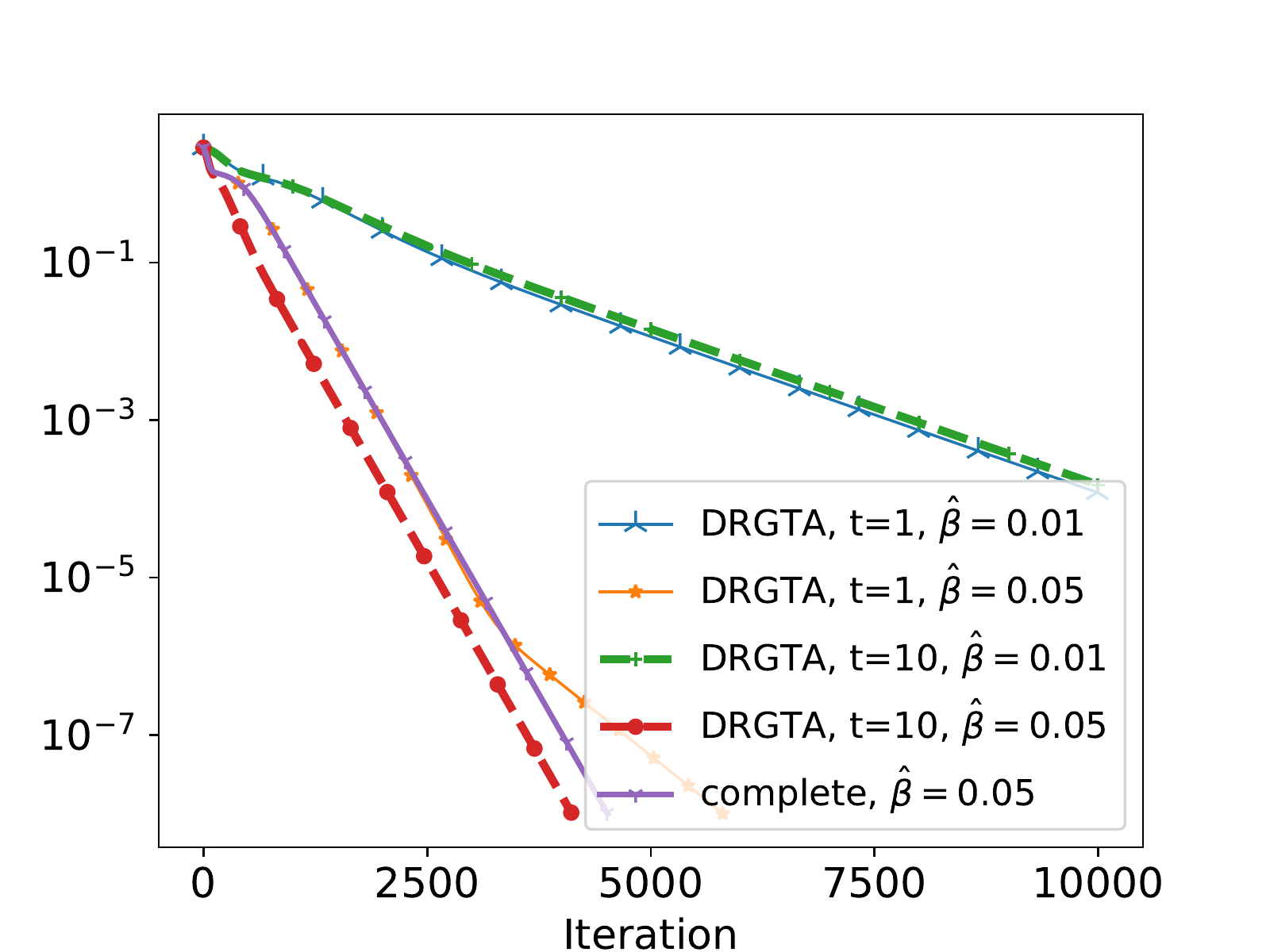}} }
		\endminipage
		\caption{Synthetic data, agents number $n=32,$ eigengap $\Delta=0.8$, Graph: Ring.  y-axis: log-scale $d_s(\bar x_k, x^*)$. } 
		\label{figure:synthetic-n-32-different start}
	\end{center}
\end{figure*}

\subsection{Real-world data}
We compare our algorithms with a recently proposed algorithm decentralized Sanger's algorithm (DSA) \cite{gang2021linearly}, which is a Euclidean-type algorithm.  To solve the eigenvector problem \eqref{prob:eigenvector}, DSA is shown to converge linearly to a neighborhood of the optimal solution. The computation  of DSA iteration is cheaper than DRDGD since there is no retraction step. For simplicity, we fix $t=1$ and $r=5$ in this section. 

We provide  some numerical results on the MNIST dataset\cite{lecun1998mnist}.
The graph is still the ring and $W$ is the Metropolis constant weight matrix. For MNIST, there are $60000$ samples and the dimension is given by $d=784.$ We normalize the data matrix by dividing $255$ such that the elements are in $[0,1].$ The data set is evenly partitioned  into $n$ subsets. The stepsizes of DRDGD and DRGTA are set to $\beta= \frac{\hat \beta}{60000}.$

The results for MNIST data set with $n=20, 40$ are shown in \cref{figure:MNIST}. We see that the convergence rate of DSA and DRDGD are almost the same and DRGTA with $\hat\beta=0.1$ can achieve the most accurate solution. When $n$ becomes larger, the convergence rate of all algorithms  is slower.  Although the computation of DSA is cheaper than DRDGD, we find that when $\hat\beta=0.5, n=20, $ DSA does not converge, which is not shown in the   \cref{figure:MNIST} (a).  This is probably because DSA is not a feasible method and needs carefully tuned stepsize. 
\begin{figure*}[ht]
	\begin{center}
		\minipage{0.45\textwidth}
		\subfigure[MNIST, $n=20$, ring graph]{
			{\includegraphics[width= \linewidth]{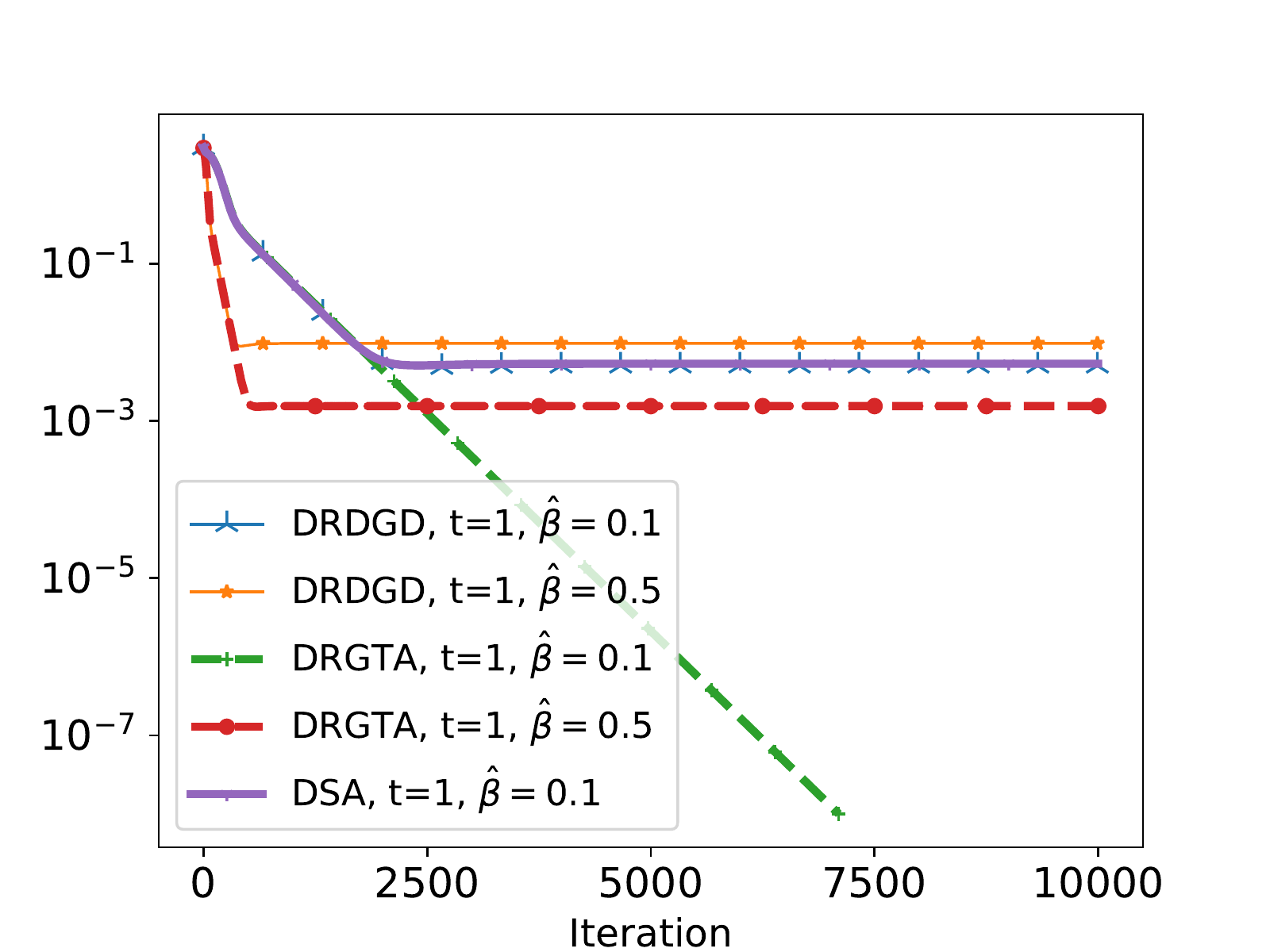}} }
		\endminipage
		\minipage{0.45\textwidth}
		\subfigure[MNIST, $n=40$, ring graph]{
			{\includegraphics[width=\linewidth]{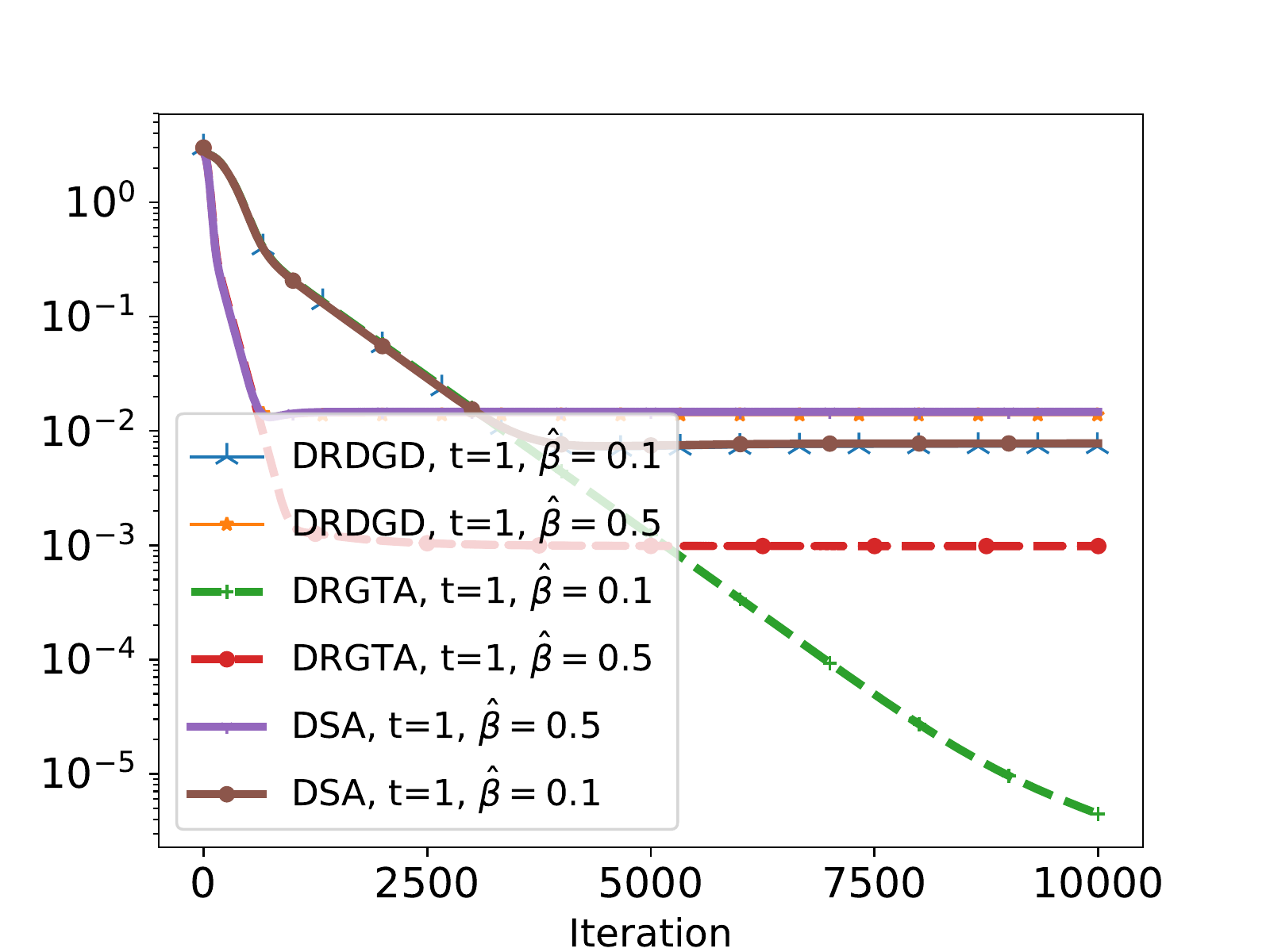}} }
		\endminipage
		\caption{Numerical results of DRDGD, DRGTA, DSA on MNIST data set.  y-axis: log-scale $d_s(\bar x_k, x^*)$. } 
		\label{figure:MNIST}
	\end{center}
\end{figure*}

Finally, we demonstrate the linear speedup of DRSGD for different $n$.  The experiments are evaluated in a HPC cluster, where each computation node  is associated with an Intel Xeon E5-2670 v2 CPU.  The  computation nodes are connected by FDR10 Infiniband.  We use $10$ CPU cores each computation node in the HPC cluster. And we treat one CPU core  as one network node in our problem.  The codes are implemented in python with mpi4py. 

We set the maximum epoch as $300$ in all experiments. The stepsize is set to $\beta = \frac{\sqrt{n}}{10000\sqrt{300}}\hat \beta,$ where $\hat \beta$ is tuned for the best performance.  The results in \cref{figure:MNIST-stochastic} are $\log d_s(\bar{x}_k, x^*)$ v.s. epoch and $\log d_s(\bar{x}_k, x^*)$ v.s. CPU time, respectively.   As we see in \cref{figure:MNIST-stochastic}(a), the solutions accuracy of $n=16, 32, 60$ are almost the same, while the CPU time in \cref{figure:MNIST-stochastic}(b) can be accelerated by   nearly linear ratio. 

\begin{figure*}[ht]
	\begin{center}
		\minipage{0.45\textwidth}
		\subfigure[iteration]{
			{\includegraphics[width= \linewidth]{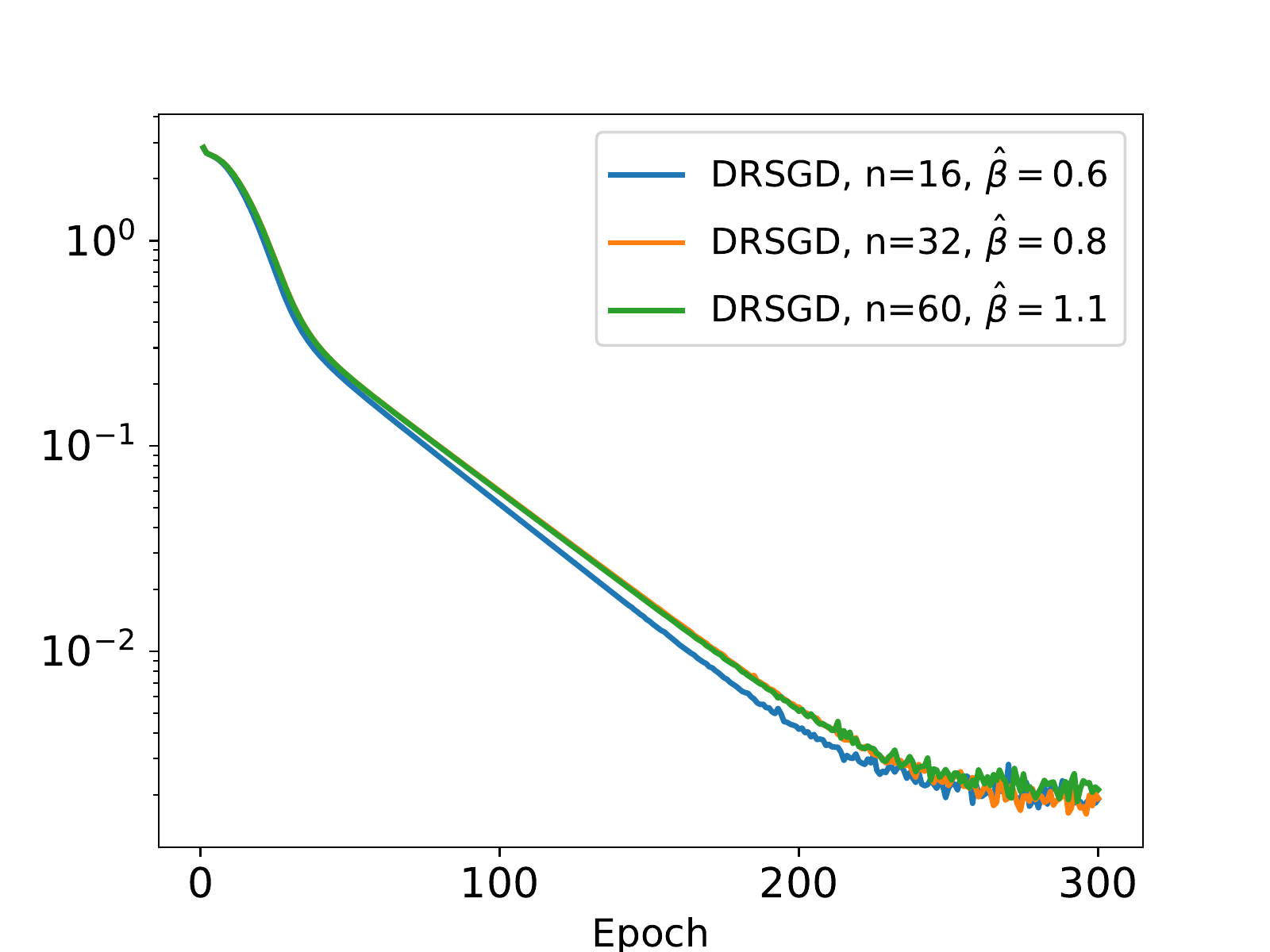}} }
		\endminipage
		\minipage{0.45\textwidth}
		\subfigure[time ]{
			{\includegraphics[width=\linewidth]{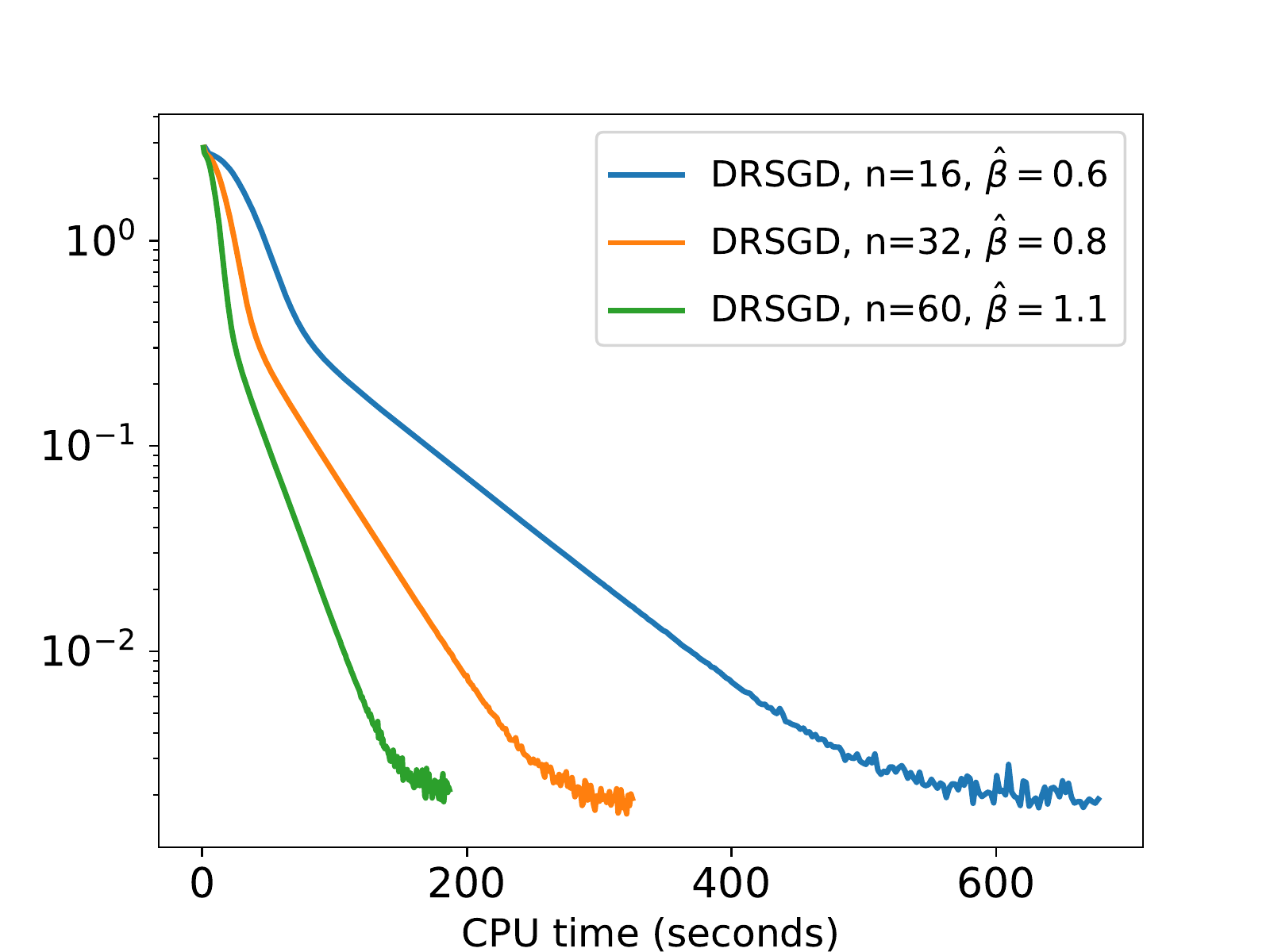}} }
		\endminipage
		\caption{Comparison results of different number of nodes on MNIST. Ring graph associated with Metropolis constant weight matrix, $t=1,$ $\beta = \frac{\sqrt{n}}{10000\sqrt{300}}\hat \beta.$ y-axis: log-scale $d_s(\bar x_k, x^*)$.  } 
		\label{figure:MNIST-stochastic}
	\end{center}
\end{figure*}

\section{Conclusions}
We discuss the decentralized optimization problem over Stiefel manifold and  propose the two decentralized Riemannian gradient method and establish their convergence rate. 
Future topics could be cast into the following several folds: Firstly, for the eigenvector problem \eqref{prob:eigenvector}, it will be interesting to establish the linear convergence of DRGTA.  Secondly, the analysis is based on the local convergence of Riemannian consensus, which results in multi-step consensus. It would be interesting to design  algorithms based on Euclidean consensus. 

\bibliography{manifold}
\bibliographystyle{plain}


\newpage

\appendix
\onecolumn
\section{About the polar retraction}\label{sec:append_retraction}
Given the polar decomposition of $x+\xi=QH,$ where $Q\in\R^{d\times r}$ is orthogonal and $H\in\R^{r\times r}$ is positive definite. 
The polar retraction is the polar factor 
\begin{align}\label{eq:polar}
    \Retr_{x}(\xi)=Q = (x+\xi)(I_r + \xi^\top \xi)^{-1/2},
\end{align}
which is also the orthogonal projection of $x+\xi$ onto $\St(d,r)$.   The computation complexity is $\mathcal{O}(dr^2).$
 \cite[Append. E]{Liu-So-Wu-2018}  showed that if $\normfro{\xi}\leq 1$ then $M=1$   for polar retraction.  The boundedness of $\xi$ can be verified in our convergence analysis. Therefore, we have $M=1$ in this paper.

\section{More details on linear rate of consensus}\label{sec:append_consensus}

The following results were provided in \cite{chen2020consensus}.

 If there exists an integer $t\geq 0$ such that 
\be\label{ineq:W mult step t}
\max_{i\in[n]}  \normfro{\sum_{j=1}^n (W_{ij}^t - 1/n )(x_j-\bar x) }\leq\max_{i\in[n]}  \sum_{j=1}^n |W_{ij}^t - 1/n |\normfroinf{\X -\bar\X} \leq \frac{1}{2}\normfroinf{\X -\bar\X},
\ee
then it suffices to show the sequence  $\{\X_k\}$ of DRCS  satisfying  $\X_k\in\N$ with $t\geq \lceil\log_{\sigma_2}(\frac{1}{2\sqrt{n}})\rceil$ steps of communication.

 Denote the smallest eigenvalue of $W^t$ by $\lambda_n(W^t),$ the constant $L_t$   is given by 
 \be\label{def:lipschitz consensus function} L_t = 1-\lambda_n(W^t) .\ee
 It is   the Lipschitz constant of $\nabla\varphi^t(\X)$. Since $L_t\in(0,2],$ if $\lambda_n(W^t)$ is unknown, one can use $L_t=2.$
  Define the second largest eigenvalue of $W^t$ by $\lambda_2(W^t)$    and \[ \mu_t=1-\lambda_2(W^t).\]
  
  The formal statement of \cref{thm:linear_rate_consensus_informal} is given as follows. 
  \begin{fact}\label{thm:linear_rate_consensus}\cite{chen2020consensus}
  Under \cref{assump:doubly-stochastic},
 let the stepsize  $\alpha$ satisfy $0<\alpha\leq \bar\alpha:=\min\{\nu \frac{\Phi}{ L_t}, 1,\frac{1}{M}\}$ and $t\geq \lceil\log_{\sigma_2}(\frac{1}{2\sqrt{n}})\rceil$, where $\nu\in[0,1],$ $\Phi=2-\delta_2^2$ and $M$ is given in \cref{lem:nonexpansive_bound_retraction}.
	The sequence $\{\X_k\}$ of \eqref{consensus_rga} achieves consensus linearly if the initialization satisfies 
	$\X_0\in \N$ defined by \eqref{delta_1_and_delta_2}. That is, we have  $\X_k\in\N$ for all $k\geq 0$ and 
	\begin{align}\label{linear rate of consensus}
	\bad
	\normfro{\X_{k+1} - \bar{\X}_{k+1}}&\leq  \normfro{\X_k - \alpha\grad\varphi^t(\X_k)  - \bar\X_k} \\
	&\leq \sqrt{1-    2(1-\nu) \alpha  \gamma_{t}}\normfro{\X_{k} - \bar{\X}_{k}},
	\ead
	\end{align}
	where $\gamma_t = (1- 4r\delta_{1}^2) (1- \frac{\delta_{2}^2}{2}) \mu_t\geq \frac{\mu_t}{2}\geq \frac{1-\sigma_2^t}{2}$. 
\end{fact}
If $\nu=1/2$, we have $\alpha\leq \bar\alpha:=\min\{\frac{\Phi}{2L_t},1,1/M\}$ and 
\[ \rho_t=\sqrt{1-\gamma_t\alpha }. \]
Recall that $M$ is the constant given in \cref{lem:nonexpansive_bound_retraction}. We also have $M=\mathcal{O}(1)$ which is discussed in \cref{sec:append_retraction}.  
If $\alpha=1$ is admissible, then the rate is $\rho_t=\sqrt{\frac{1+\sigma_2^t}{2}}$ which is worse that the Euclidean rate $\sigma_2^t.$ Moreover, it was shown in \cite{chen2020consensus} in a smaller region, i.e., $\varphi^t(\X)=\mathcal{O}(\sigma_2^t)$ and $\normfro{\X-\bar\X}^2=\mathcal{O}(1), $ it follows asymptotically $\rho_t=\sigma_2^t$ with $\alpha=1$. For simplicity, we will only discuss the convergence of our proposed algorithms using \eqref{linear rate of consensus} with $\nu=1/2$. 
Note that this may imply $\bar\alpha<1$, but we find that $\alpha=1$ always works for our proposed algorithms.

\section{Proofs for \cref{sec:preliminary}}\label{sec:append-preliminary}
Denote $\p_{N_{x}\M}$ as the orthogonal projection onto the normal space $N_x\M$.  One can rewrite the projection  $	\p_{\T_x\M}(y - x), \forall y\in\St(d,r)$ \cite{chen2020consensus} as follows
	\be\label{property of projection onto tangent}
	\bad
  &\quad  \p_{\T_x\M}(y - x) = y -  x - \p_{N_{x}\M}(y-x) \\
& = y- x + \half x (x-y)^\top (x-y).
	\ead  \tag{P2}
  \ee
 This   implies that
\[  \p_{\T_x\M}(y - x) = y-x + \mathcal{O}(\normfro{y-x}^2). \] 
The relationship \eqref{property of projection onto tangent} helps us to prove \cref{lem:lipschitz}. 
\begin{proof}[\textbf{Proof of \cref{lem:lipschitz}}]
 
	Firstly, since $\nabla f(x)$ is $L-$Lipschitz in Euclidean space, one has
	\be\label{ineq:lip_smooth_E}
	\left|f(y) - \left[ f(x) + \inp{\nabla f(x)}{y-x} \right]  \right|\leq \frac{L}{2}\normfro{y-x}^2.
	\ee	
Since $\grad f(x)=\p_{\T_x\M}\nabla f(x)$, we have
\begin{align}
   \inp{\grad f(x)}{y-x}  
	&=   \inp{\nabla f(x)}{\p_{\T_{x}\M}(y-x)}   \notag\\
	&\stackrel{\eqref{property of projection onto tangent}}{=} \inp{\nabla f(x)}{y-x} + \inp{ \nabla f(x)}{\half x (y-x)^\top (y-x) }. \notag 
	\end{align}	
	Using
	 \[  \inp{ \nabla f(x)}{\half x (y-x)^\top (y-x) }\leq  \normtwo{\nabla f(x)}\cdot \normtwo{x}\cdot  \half\normfro{ x-y }^2\leq  \half \normtwo{\nabla f(x)} \cdot  \normtwo{ x-y }^2\] 
	 implies
	 \be\label{lips-proof-1}
	 \left| \inp{\grad f(x)}{y-x}  -  \inp{\nabla f(x)}{y-x} \right| \leq \half\max_{x\in\St(d,r)}\normtwo{\nabla f(x)}\cdot\normfro{y-x}^2,
	 \ee
	where $\normtwo{\nabla f(x)}$ represents  the operator norm of  $\nabla f(x)$. 
	 Since $\St(d,r)$ is a compact set and $\nabla f(x)$ is continuous, we denote $L_n =\max_{x\in\St(d,r)}\normtwo{\nabla f(x)}$. 
	Let $L_g =  L_n + L$. Combining \eqref{ineq:lip_smooth_E} with \eqref{lips-proof-1} yields
	\be\label{ineq:lip_smooth_R}
	\left|f(y) - \left[ f(x) + \inp{\grad f(x)}{y-x} \right]  \right|\leq \frac{L_g}{2}\normfro{y-x}^2.
	\ee	
   Secondly, using $\grad f(x)= \nabla f(x) - \p_{N_x\M}\nabla f(x)$ and $\grad f(y)= \nabla f(y) - \p_{N_y\M}\nabla f(y)$ implies
   \begin{align}
       &\quad \normfro{\grad f(x) - \grad f(y)} \notag\\
       &\leq  \normfro{ \nabla f(x) -  \nabla f(y)} + \normfro{\p_{N_x\M}\nabla f(y) - \p_{N_y\M}\nabla f(y)} \notag\\
       & = \normfro{ \nabla f(x) -  \nabla f(y)} +\half \normfro{ x (x^\top \nabla f(y) + \nabla f(y)^\top x) - y(y^\top \nabla f(y) + \nabla f(y)^\top y)}\notag \\
       &\leq \normfro{\nabla f(x) - \nabla f(y)} + 2L_n\normfro{x-y}\label{ineq:lips-proof-2}\\
       &\leq (L+2L_n)\normfro{x-y}.\notag
   \end{align}
   In \eqref{ineq:lips-proof-2} we used
   \begin{align}
         &\quad\normfro{ x (x^\top \nabla f(y) + \nabla f(y)^\top x) - y(y^\top \nabla f(y) + \nabla f(y)^\top y)}\notag\\
        &\leq \normfro{ x ((x-y)^\top \nabla f(y) + \nabla f(y)^\top (x-y))} + \normfro{  (x- y)(y^\top \nabla f(y) + \nabla f(y)^\top y)}\notag\\
        &\leq 4L_n\normfro{x-y}.\notag
   \end{align}
 
	  The proof is completed. 
\end{proof}
\subsection{Comparison on different Lipschitz-type inequalities}\label{subsection:append-comparison lipschitz ineq}
Using Taylor's Theorem\cite[Lemma 7.4.7]{Absil2009}, $L_g'$ corresponds to the leading eigenvalue of Riemannian Hessian.  According to \cite{absil2013extrinsic}, it follows for any $\eta\in\T_x\M$ that
\be\label{Riemannian hessian}
\bad
\Hess f(x)[\eta] &=  \p_{T_{x\M}}\left( D\grad h(x)[\eta] \right)\\
& = \p_{T_{x}\M}\nabla^2 f(x)  \eta  -\eta x^\top \p_{N_{x}} \nabla f(x) - x\half \left(\eta^\top\p_{N_{x}} \nabla f(x_i) + (\p_{N_{x}} \nabla f(x)) ^\top \eta \right),
\ead
\ee
where $\p_{N_{x}}$ is the orthogonal projection onto the normal space $N_x\M$.  Since   $x\half \left(\eta^\top\p_{N_{x}} \nabla f(x_i) + (\p_{N_{x}} \nabla f(x)) ^\top \eta \right) \in  N_x\M$, we have
\be\label{Riemannian hessian-1}
\bad
\inp{\eta}{ \Hess f(x)[\eta]}  
=  \inp{\eta}{  \nabla^2 f(x)  \eta }  - \inp{\eta}{\eta x^\top \p_{N_{x}} \nabla f(x) } = \inp{\eta}{  \nabla^2 f(x)  \eta }  - \inp{\eta}{\eta \half (x^\top \nabla f(x) + \nabla f(x)^\top x)}  ,
\ead
\ee
where we use $\p_{T_{x}\M}  \nabla^2 f(x)  \eta =  \nabla^2 f(x)  \eta - \p_{N_{x}}  \nabla^2 f(x)  \eta$. Therefore, we get 
\be\label{ineq:L-g}
L_g^\prime \leq \lambda_{\max} (\nabla^2 f(x)) + \max_{x\in\St(d,r)}\normtwo{ \nabla f(x)}= L+ L_n.
\ee
 
 The restricted inequality proposed in \cite{boumal2019global}   is related to the pull back function $g(\xi):=f(\Retr_{x}(\xi))$, whose Lipschitz constant $\tilde{L}_g$ relies on the retraction.
Specifically,   $\tilde{L}_g= M_0^2 L + 2ML_n$, where $ M_0$ is a constant related to the retraction, $M$ and $L_n$ are the same  constants in \cref{lem:nonexpansive_bound_retraction}.


\subsection{Technical lemmas}
 
\begin{lemma}\cite{chen2020consensus}\label{lem:distance between Euclideanmean and IAM}
	For any $\X\in\St(d,r)^n$, let $\hat x  =\frac{1}{n}\sum_{i=1}^n x_i$ be the Euclidean mean and denote $\hat\X = \mathbf{1}_n\otimes \hat x$. Similarly, let $\bar\X = \mathbf{1}_n\otimes \bar x$, where $\bar{x}$ is the IAM defined in \eqref{eq:IAM}.  
	Moreover, if $\normfro{\X - \bar\X}^2\leq n/2$, one has  
	\begin{equation}\label{key}
	\normfro{\bar x - \hat x }\leq \frac{2\sqrt{r}\normfro{\X-\bar\X}^2}{n}. \tag{P1}
	\end{equation}
\end{lemma}
The following lemma will be useful to bound the Euclidean distance between two average points $\bar{x}_k$ and $\bar{x}_{k+1}$.

\begin{lemma} \cite{chen2020consensus}\label{lem:bound_of_two_average}
	Suppose $\X, \Y\in \N_1$, where $\N_1$ is defined in \eqref{contraction_region_1}.    
	Then  we have	\[\normfro{\bar{x} - \bar{y}}\leq \frac{1}{1-2\delta_1^2} \normfro{\hat x - \hat y},  \]
	where $\bar{x}$ and $\bar{y}$ are the IAM of $x_1,\ldots,x_n$ and $y_1,\ldots,y_n$, respectively. 
\end{lemma}

We also need the following bounds for $\grad \varphi^t(\X)$.  

\begin{lemma}\cite{chen2020consensus}\label{lem:bound_of_grad}
	For any $\X\in\St(d,r)^n$, it follows that
	\be\label{ineq:bound-of-sum-gradh}
	\begin{aligned}
		\normfro{ \sum_{i=1}^{n}\grad \varphi^t(x_{i})}
		\leq L_t\normfro{\X -\bar \X}^2 
	\end{aligned} \ee
	and
	\be\label{ineq:bound-of-gradh} \normfro{\grad \varphi^t(\X)}\leq  L_t\normfro{\X - \bar\X},\ee
	where $L_t$ is the  constant given in \eqref{def:lipschitz consensus function}. Moreover, suppose $\X\in\N_{2}$, where $\N_{2}$ is defined by \eqref{contraction_region_2}. We then have
	\be\label{ineq:bound-of-gradh-i}  \max_{i\in[n]}\normfro{\grad \varphi^t_i(\X)}\leq 2 \delta_{2} . \ee	
\end{lemma}

Applying \cref{lem:bound_of_two_average} to the update rule of our algorithms gives the following lemma. 

\begin{lemma}\label{lem:bound_of_k_k+1}
	If $\X_k\in\N, \X_{k+1}\in\N_1$ and  $x_{i,k+1} = \Retr_{x_{i,k}}(-\alpha \grad \varphi_i^t(\X_{k}) + \beta u_{i,k})$, where $u_{i,k}\in\T_{x_{i,k}}\M$, $0\leq \alpha  \leq \frac{1}{M}$,  $0\leq \beta $. 	Let $\U_k^\top=(u_{1,k}^\top \ \ldots \ u_{n,k}^\top)$ and $\hat{u}_k=\frac{1}{n}\sum_{i=1}^{n}u_{i,k}$.  It follows that 
	\[\normfro{\bar x_k - \bar{ x}_{k+1}}\leq  \frac{1}{1-2\delta_1^2}\left(  \frac{2L_t^2 \alpha +  L_t\alpha   }{n}  \normfro{\X_k - \bar\X_k}^2 + {\beta} \normfro{\hat u_k}+ \frac{ 2M\beta^2}{n}\normfro{\U_k}^2 \right).\]
\end{lemma}

\begin{proof}
 From \cref{lem:nonexpansive_bound_retraction} and \cref{lem:bound_of_grad}, we  have 
\begin{align*}
&\quad \normfro{\hat x_k - \hat x_{k+1}} \\
&\leq  \normfro{\hat x_k+ \frac{1}{n}\sum_{i=1}^{n} (-\alpha \grad \varphi_i^t(\X_{k}) +\beta u_{i,k}) -\hat x_{k+1} } + \normfro{ \frac{1}{n}\sum_{i=1}^{n} (-\alpha \grad \varphi_i^t(\X_{k}) +\beta u_{i,k})}  \\
&\stackrel{\eqref{ineq:ret_second-order}}{\leq } \frac{M}{n} \sum_{i=1}^{n} \normfro{\alpha  \grad \varphi_i^t(\X_{k}) +\beta u_{i,k} }^2 + \alpha \normfro{ \frac{1}{n}\sum_{i=1}^{n}  \grad \varphi_i^t(\X_{k})} +\beta \normfro{\hat u_k}\\
&\leq \frac{ 2M\alpha^2}{n}\normfro{\grad \varphi^t(\X_k)}^2 + \frac{ 2M\beta^2}{n}\normfro{\U_k}^2 + \alpha \normfro{ \frac{1}{n}\sum_{i=1}^n \grad \varphi_i^t(\X_{k})} +\beta \normfro{\hat u_k}\\ 
&\stackrel{\eqref{ineq:bound-of-gradh}\eqref{ineq:bound-of-sum-gradh}}{\leq }  \frac{ 2L_t^2M\alpha^2 + L_t\alpha }{n}  \normfro{\X_k - \bar\X_k}^2 + \frac{ 2M\beta^2}{n}\normfro{\U_k}^2 + \beta \normfro{\hat u_k}.
\end{align*}
Therefore, it follows from \cref{lem:bound_of_two_average} that
\begin{align*}
\normfro{\bar x_k - \bar{ x}_{k+1}} \leq \frac{1}{1-2\delta_1^2} \normfro{\hat x_k - \hat x_{k+1}} \leq  \frac{1}{1-2\delta_1^2}\left(  \frac{2L_t^2 \alpha +  L_t\alpha   }{n}  \normfro{\X_k - \bar\X_k}^2 + {\beta} \normfro{\hat u_k}+ \frac{ 2M\beta^2}{n}\normfro{\U_k}^2 \right),
\end{align*}
where we use the fact that $\alpha \leq \frac{1}{M}$.
\end{proof}

\section{Proofs for \cref{sec:distributed gradient method}}
 	We use the notations \[ \V_k = [ v_{1,k}^\top \ \ldots \ v_{n,k}^\top ]^\top, \quad \hat v_k = \frac{1}{n}\sum_{i=1}^n v_{i,k},\]
\[ g_{i,k} = \grad f_i(x_{i,k})\quad  \textit{and}\quad \hat g_k = \frac{1}{n}\sum_{i=1}^n g_{i,k}.\]
The following lemma is useful to show $\X_k\in\N$ for all $k.$
\begin{lemma}\label{lem:total deviation from 1/N}\cite[Lemma 11]{chen2020consensus}
	Given any $\X\in\N_{2}$, where $\N_{2}$ is defined in \eqref{contraction_region_2}, 
	if $t\geq \lceil \log_{\sigma_2}(\frac{1}{2\sqrt{n}})\rceil$, we have
	\be\label{lem:bound_of_2_infty} \max_{i\in[n]} \normfro{\sum_{j=1}^n (W^t_{ij}-1/n)x_j}\leq \frac{\delta_{2}}{2}.  \ee
\end{lemma}

\begin{lemma} \label{lem:recursive lemma}
	Under the same conditions of \cref{thm:linear_rate_consensus}, if $\X_k\in\N$ and  \[ x_{i,k+1} = \Retr_{x_{i,k}}(-\alpha \grad \varphi_i^t(\X_{k}) + \beta v_{i,k}), \quad \forall i\in[n],\] where $v_{i,k}\in\T_{x_{i,k}}\M$, 
  the following holds
 \[ \normfro{\X_{k+1}-\bar{\X}_{k+1}} 
	 \leq \rho_t\normfro{\X_k - \bar \X_k} + \beta_k \normfro{\V_k}.\]
\end{lemma}
\begin{proof}
	By the definition of IAM, we have
	\be\label{ineq:de-1}
	\begin{aligned}
	\quad &\normfro{\X_{k+1}-\bar{\X}_{k+1}}^2\leq	\normfro{\X_{k+1}-\bar{\X}_{k}}^2 \\  	
	= &\sum_{i=1}^n \normfro{\Retr_{x_{i,k}}\left(- \alpha\grad \varphi_i^t(\X_{k}) -\beta_k v_{i,k}\right) - \bar{x}_{k}}^2\\
	\stackrel{\eqref{ineq:ret_nonexpansive}}{\leq }& \sum_{i=1}^n \normfro{x_{i,k} -\alpha\grad \varphi_i^t(\X_{k})   -\beta_k v_{i,k} -  \bar{x}_{k}}^2.
	\end{aligned}\ee
	   	Let $\V_k = [ v_{1,k}^\top \ \ldots \ v_{n,k}^\top ]^\top $.
     Then, we get
	 \be\label{ineq:de-2}
	 \bad
	 \normfro{\X_{k+1}-\bar{\X}_{k+1}}&\leq \normfro{\X_k - \alpha\grad\varphi^t(\X_k) - \beta_k \V_k - \bar\X_k}\\
	 &\leq  \normfro{\X_k - \alpha\grad\varphi^t(\X_k)  - \bar\X_k} + \beta_k \normfro{\V_k}.
    \ead\ee
    By combining inequality \eqref{linear rate of consensus} of \cref{thm:linear_rate_consensus}, we get 
    	 \be\label{ineq:de-3}
	 \bad
	 \normfro{\X_{k+1}-\bar{\X}_{k+1}} 
	 &\leq \rho_t\normfro{\X_k - \bar \X_k} + \beta_k \normfro{\V_k}.
    \ead\ee
    The proof is completed. 
 
\end{proof}

\begin{proof}[\textbf{Proof of \cref{lem:convergence_of_deviation from mean}} ]
   We prove that $\X_k\in\N$ for all $k\geq 0$ by induction. Suppose $\X_k\in\N$, let us show $\X_{k+1}\in\N.$
	Note $\normfro{\V_k}\leq \sqrt{n}D.$ Using \cref{lem:recursive lemma} yields 
	\begin{align}
	\normfro{\X_{k+1}-\bar{\X}_{k+1}}  
	&\leq   \rho_t\normfro{\X_{k} -  \bar{\X}_{k}  } + \beta_k\sqrt{n}D \label{ineq:linear_consensus_grad_0}\\
	&\leq  \rho_t \sqrt{n}\delta_1 + \beta_k\sqrt{n}D \notag\\
	&\leq \sqrt{n}\delta_1,\notag
	\end{align}
	where the last inequality follows from $\beta_k\leq \frac {1-\rho_t}{ D} \delta_1$.  Hence $\X_{k+1}\in\N_1$.
	Secondly, let us verify $\X_{k+1}\in\N_2$.   It follows  from $\beta_k\leq \frac{\alpha\delta_1}{5D}\leq \frac{\alpha}{2D}$ and $\alpha\leq 1$  that  
	 \[\normfroinf{\X_{k+1}-  {\X}_k} \stackrel{\eqref{ineq:ret_nonexpansive}}{\leq} \max_{i\in[n]}\normfro{\alpha\grad \varphi^t(x_{k,i})}+ \beta_k D\stackrel{\eqref{ineq:bound-of-gradh-i}}{\leq} 2\alpha\delta_2 +  \frac{\alpha }{2} \leq 1-\delta_1^2.\]
	Using \cref{lem:bound_of_k_k+1} yields
		\begin{align*}\normfro{\bar x_k - \bar{ x}_{k+1}}&\leq \frac{1}{1-2\delta_1^2}\left(  \frac{2L_t^2 \alpha +  L_t\alpha   }{n}  \normfro{\X_k - \bar\X_k}^2 + {\beta} \normfro{\hat v_k}+ \frac{ 2M\beta_k^2}{n}\normfro{\V_k}^2 \right)\\
		&\leq \frac{1}{1-2\delta_1^2}  \left[  (2L_t^2\alpha +L_t\alpha)\delta_1^2+ \beta_k D  + 2M\beta_k^2 D^2 \right].\end{align*}
		 	Furthermore, since   $L_t\leq 2$,  $\beta_k\leq \frac {\alpha\delta_1}{5D}$, $\alpha \leq 1/M$, we get
		\be\label{ineq:de0} \normfro{\bar x_k - \bar{ x}_{k+1}} \leq  \frac{1}{1-2\delta_1^2} \left( \frac{252}{25}\alpha \delta_1^2 + \frac{\alpha \delta_1}{5} \right)  \leq   \frac{1}{1-2\delta_1^2} \left( \frac{252}{625r} \alpha\delta_2^2   + \frac{ 1}{25\sqrt{r}}\alpha \delta_2 \right),  \ee
		where the last inequality follows from    $\delta_1\leq \frac{1}{5\sqrt{r}}\delta_2$.
	Then, one has 
	\begin{align}
	& \normfro{x_{i,k+1} - \bar{x}_{k+1 }}  \notag \\
	\leq &  \normfro{x_{i,k+1} - \bar{x}_{k}}+\normfro{\bar x_k - \bar{ x}_{k+1}} \notag\\
	\stackrel{\eqref{ineq:ret_nonexpansive}}{\leq} & \normfro{x_{i,k} - \alpha \grad \varphi_i^t(\X_{k}) - \beta_k v_{i,k}  -\bar{x}_{k } }+\normfro{\bar x_k - \bar{ x}_{k+1}}\notag\\
	\leq & \normfro{x_{i,k} - \alpha \grad \varphi_i^t(\X_{k})  -\bar{x}_{k } }+\frac{1}{5}\alpha \delta_1 +\normfro{\bar x_k - \bar{ x}_{k+1}}. \label{ineq:de1}
	\end{align}
	Now,  we proceed by using the same lines in the proof of \cite[Lemma 13]{chen2020consensus} as follows
	\be\label{rewrite}
	\grad \varphi^t_i(\X) = x_i - \sum_{j=1}^n W_{ij}x_j - \half x_i  \sum_{j=1}^n W_{ij}^t(x_{i}-x_{j})^\top (x_{i}-x_{j}),
	\ee
	and 
	\begin{align}
	    \quad &\normfro{x_{i,k} - \alpha \grad \varphi_i^t(\X_{k})  -\bar{x}_{k } }\notag\\
	    	\stackrel{\eqref{rewrite}}{=} &  \normfro{ (1-\alpha)(x_{i,k} - \bar{x}_k) + \alpha (\hat{x}_k- \bar{x}_{k })+\alpha\sum_{j=1}^n{W}_{ij}^t (x_{j,k} - \hat{x}_k) 
	+ \frac{\alpha}{2}x_{i,k}\sum_{j=1}^n W_{ij}^t(x_{i,k}-x_{j,k})^\top(x_{i,k}-x_{j,k})   }\notag\\
	\leq &(1-\alpha)\delta_{2} + \alpha \normfro{   \hat x_{k} - \bar{x}_k }
	 +\alpha \normfro{ \sum_{j=1}^n ({W}_{ij}^t -\frac{1}{n}) x_{j,k}  } 
	+   \frac{1}{2}\normfro{ {\alpha} \sum_{j=1}^n W_{ij}^t(x_{i,k}-x_{j,k})^\top(x_{i,k}-x_{j,k}) }\label{ineq:key0}\\ 
	\leq & (1-\alpha)\delta_{2} + 2\alpha   \delta_{1,t}^2\sqrt{r}    +\alpha \normfro{ \sum_{j=1}^n ({W}_{ij}^t -\frac{1}{n}) x_{j,k} }   +2{\alpha}\delta_{2}^2 \label{ineq:key1}\\
	 {\leq } &(1-\frac{\alpha}{2})\delta_{2} + 2 \alpha   \delta_{1}^2 \sqrt{r}  + 2{\alpha}\delta_{2}^2,\label{ineq:key2}
	\end{align}
	where \eqref{ineq:key0} follows from $\alpha\in[0,1]$, \eqref{ineq:key1} holds by \cref{lem:distance between Euclideanmean and IAM} and \eqref{ineq:key2} follows from \cref{lem:total deviation from 1/N}. Combining this with \eqref{ineq:de1} implies
		\begin{align}
	& \normfro{x_{i,k+1} - \bar{x}_{k+1 }}  \notag \\
	\leq &   (1-\frac{\alpha}{2})\delta_{2} + 2 \alpha   \delta_{1}^2 \sqrt{r}  + 2{\alpha}\delta_{2}^2  + \frac{1}{5}\alpha \delta_1 +\normfro{\bar x_k - \bar{ x}_{k+1}}\notag\\
	\stackrel{\eqref{ineq:de0}}{\leq}&   (1-\frac{\alpha}{2})\delta_{2} + 2 \alpha   \delta_{1}^2 \sqrt{r}  + 2{\alpha}\delta_{2}^2  +\frac{1}{5}\alpha \delta_1 + \frac{1}{1-2\delta_1^2} \left( \frac{252}{625r} \alpha\delta_2^2   + \frac{ 1}{25\sqrt{r}}\alpha \delta_2 \right).\label{ineq:de3}
	\end{align}
	Therefore, substituting the conditions   \eqref{delta_1_and_delta_2} on $\delta_1, \delta_2$  into \eqref{ineq:de3} yields
	\begin{align*} \normfro{x_{i,k+1} - \bar{x}_{k+1 }}\leq \delta_2.
	\end{align*}
	The proof of the first statement is completed. 
	Finally, it follows from \eqref{ineq:linear_consensus_grad_0} that
	\be\label{ineq:linear_consensus_grad}
	\bad
 \normfro{\X_{k+1}-\bar{\X}_{k+1}}  
	&\leq   \rho_t\normfro{\X_{k} -  \bar{\X}_{k}  } + \beta_k\sqrt{n}D  \\
	&\leq   \rho_t^{k+1} \normfro{\X_0 - \bar\X_0}  +    \sqrt{n} D \sum_{l=0}^k \rho_t^{k-l}\beta_l.
	\ead
	\ee
\end{proof}

An immediate result of \cref{lem:convergence_of_deviation from mean} is that the rate of consensus $\normfro{\X_k-\bar\X_k}^2 = \mathcal{O}(\beta_k^2)$ if $\beta_k = \mathcal{O} (\frac{1}{k^p})$.   The proof is similar as \cite[Proposition 8]{liu2017convergence}, we provide it  for completeness. 
\begin{lemma}\label{coro:rate_of_consensus}
	Under \cref{assump:lips in Rn,assump:doubly-stochastic,assum:stochastic-grad,assump:stepsize_beta},
	for   \cref{alg:DRPG}, if $\X_0\in\N$,   $0<\alpha\leq\min\{ \frac{\Phi}{ 2L_t}, 1,\frac{1}{M}\}$, $t\geq \lceil\log_{\sigma_2}(\frac{1}{2\sqrt{n}})\rceil$ and
	\be\label{cond:addtional_cond_on beta} \beta_k = \min\{  \frac {\alpha\delta_1   }{5D} \cdot \frac{1}{(k+1)^p},  \frac {1-\rho_t}{ D} \delta_1\}, \quad p\in(0,1],\ee then there exists a constant $C>0$ such that $\frac{1}{n}\normfro{\X_k-\bar\X_k}^2 \leq CD^2\beta_k^2$ for any $k\geq 0$, where $C$ is independent of $D$  and $n$.
\end{lemma}

\begin{proof}[\textbf{Proof of \cref{coro:rate_of_consensus}}] 
	The proof relies on  \cref{lem:convergence_of_deviation from mean}. Let $a_k: = \frac{\normfro{\X_k-\bar\X_k}}{ \sqrt{n}\beta_k}$. 
	 
	It follows from \eqref{ineq:linear_consensus_grad} that
	\be\label{ineq:linear_consensus_grad_1}
	\bad
	a_{k+1} &\leq \rho_t a_k +  D\cdot\frac{\beta_k}{\beta_{k+1}}\leq \rho_t^{k+1-K} a_K +  D \sum_{l=K}^k\rho_t^{k-l} \frac{\beta_l}{\beta_{l+1}}.
	\ead
	\ee
	Recall that $\beta_k=\mathcal{O}(1/D)$ and $\frac{1}{n}\normfro{\X_0-\bar\X_0}^2\leq \delta_1^2$, it follows that $a_0\leq \delta_1 /\beta_0=\mathcal{O}(D)$. Since $\lim_{k\rightarrow \infty}\frac{\beta_{k+1}}{\beta_k}=1$, there exists sufficiently large $K$ such that 
	\[ \frac{\beta_k}{\beta_{k+1}}\leq 2,\quad  \forall k \geq K. \]
	For $0\leq k\leq K$, there exists some $C'>0$ such that \[ a_k^2\leq C' D^2, \] where   $C'$ is independent of  $D$ and $n$. 
	For $k\geq K$, using \eqref{ineq:linear_consensus_grad_1} gives $a_k^2\leq CD^2$, where $C = 2C'  +  \frac{8}{(1-\rho_t)^2} $. Hence, we get $ \normfro{\X_k-\bar\X_k}^2 /n \leq CD^2 \beta_k^2 $ for all $k\geq 0$, where $C=\mathcal{O}(\frac{1}{(1-\rho_t)^2})$. 
\end{proof}

\begin{lemma} \label{lem:decrease lemma}
	Under \cref{assump:lips in Rn,assump:doubly-stochastic,assum:stochastic-grad,assump:stepsize_beta}, suppose $\X_k\in\N$,   $t\geq \lceil\log_{\sigma_2}(\frac{1}{2\sqrt{n}})\rceil$,  $0<\alpha\leq\min\{ \frac{\Phi}{ 2L_t}, 1,\frac{1}{M}\}$. If $x_{i,k+1} = \Retr_{x_{i,k}}(-\alpha \grad \varphi^t(x_{i,k}) - \beta_k v_{i,k})$, $0<\beta_k\leq    \min\{ \frac{1}{5L_g}, \frac{\alpha \delta_1 }{5D} \}  $ and $\beta_{k}\geq \beta_{k+1}$, where $v_{i,k} $ satisfies \cref{assum:stochastic-grad} and $L_g$ is given in \cref{lem:lipschitz}. It follows that
	\be\label{ineq:main-recursion}\bad
  &	\E_kf(\bar x_{k+1})\leq f(\bar{x}_k)  - \frac{\beta_k }{4} \normfro{\hat g_k}^2  - \frac{\beta_k}{4}\normfro{\grad f(\bar x_k)}^2\\
  &\quad + \frac{3L_g\Xi^2}{2 n} \beta_k^2 
	+    (\frac{CD^2 L_G^2}{2} + \cT_1D^4)\beta_k^3  + \cT_2 L_gD^4\beta_k^4,
	\ead\ee
	where $L_G$ is given in \cref{lem:lipschitz}, $C$ is given in \cref{coro:rate_of_consensus}, $\cT_1=  2(4\sqrt{r} + 6 \alpha)^2 C^2+ 8 M^2$ and $\cT_2 = 201\alpha^2C^2 +  9M^2 . $
\end{lemma}

Note the variance term is in the order of $\mathcal{O}(\frac{\Xi^2}{n}\beta_k^2)$, since the gradient batch size is $n$.
 
\begin{proof}[\textbf{Proof of \cref{lem:decrease lemma}}]
	Denote the conditional expectation $\E_{i,k} v_{i,k} : = \E[v_{i,k}|x_{i,k}]$ and $\E_k :=\E[\cdot |\X_k]$. 
		By invoking \cref{lem:lipschitz}, we have
		\be\label{ineq:derivation_lem_conv_rgd_00} 
	\begin{aligned}
	&\quad	\E_kf(\bar x_{k+1}) \leq f(\bar{x}_k) +  \inp{\grad f(\bar x_k)}{\E_k\bar x_{k+1} - \bar x_k} +  \frac{L_g}{2}\E_k\normfro{\bar x_{k+1} - \bar x_k}^2\\ 
	& = f(\bar{x}_k)  - \inp{\grad f(\bar x_k)}{\beta_k \hat g_k} +  \inp{\grad f(\bar x_k)}{\E_k[\bar x_{k+1} - \bar x_k + \beta_k \hat v_k]}+  \frac{L_g}{2}\E_k\normfro{\bar x_{k+1} - \bar x_k}^2\\
	 &=  f(\bar{x}_k) - \frac{\beta_k}{2}\normfro{\grad f(\bar x_k)}^2  - \frac{\beta_k}{2}\normfro{\hat g_k}^2 + \frac{\beta_{k}}{2}\normfro{\grad f(\bar x_k)-\hat g_k}^2 \\
	 &\quad+\inp{\grad f(\bar x_k)}{\E_k \bar x_{k+1} - \bar x_k + \beta_k \hat g_k}+  \frac{L_g}{2}\E_k\normfro{\bar x_{k+1} - \bar x_k}^2,
	\end{aligned}\ee
where $\hat v_k = \frac{1}{n}\sum_{i=1}^n v_{i,k}$ and  we use $\E_k \hat v_k = \hat g_k$ in the first equation.

	Note that for $\beta_k>0$, we have
	\begin{align*}
	 &\quad \inp{\grad f(\bar x_k)}{\E_k \bar x_{k+1} - \bar x_k + \beta_k \hat g_k}\leq  \frac{\beta_k}{4}\normfro{\grad f(\bar x_k)}^2  +  \frac{1}{ \beta_k}\normfro{\E_k \bar x_{k+1} - \bar x_k + \beta_k \hat g_k }^2.
	\end{align*} Plugging this into \eqref{ineq:derivation_lem_conv_rgd_00} yields
	\be\label{ineq:derivation_lem_conv_rgd_01} 
	\begin{aligned}
		&\quad \E_kf(\bar x_{k+1})\\
		&\leq  f(\bar{x}_k)- \frac{\beta_k}{2}\normfro{\hat g_k}^2  - \frac{\beta_k}{4}\normfro{\grad f(\bar x_k)}^2 + \frac{\beta_{k}}{2}\underbrace{\normfro{\grad f(\bar x_k)-\hat g_k}^2 }_{:=a_1}
		+\frac{1}{ \beta_k}\underbrace{\normfro{\E_k [\bar x_{k+1} - \bar x_k + \beta_k \hat v_k] }^2}_{:=a_2} \\
		&\quad +  \frac{L_g}{2}\underbrace{\E_k\normfro{\bar x_{k+1} - \bar x_k}^2}_{:=a_3}.
	\end{aligned}
	\ee
 Using \cref{lem:lipschitz} implies
 \begin{align*}
 	a_1\leq \frac{1}{n}\sum_{i=1}^n\normfro{\grad f(x_{i,k}) - \grad f(\bar x_k)}^2\stackrel{\eqref{ineq:lips_riemanniangrad}}{\leq} \frac{L_G^2}{n}\normfro{\X_k-\bar\X_k}^2.
 \end{align*}
 Secondly, we use the following inequality to derive the upper   bound of $a_2$.
 	  From \cref{lem:convergence_of_deviation from mean}, we have $\X_{k+1}\in\N$. One has
 \be\label{ineq:derivation_lem_conv_rgd_02} 
\begin{aligned}
	&\quad  \normfro{\bar x_{k+1} - \bar x_k + \beta_k \hat v_k} \\
	&\leq \normfro{\bar x_k - \hat x_k} + \normfro{\bar x_{k+1} - \hat x_{k+1}} + \normfro{\hat x_k - \beta_k \hat v_k - \hat x_{k+1} } \\
	&\stackrel{\eqref{key}}{\leq }\frac{2\sqrt{r}}{n}(\normfro{\X_k-\bar\X_k}^2 + \normfro{\X_{k+1}-\bar\X_{k+1}}^2)+ \normfro{\hat x_k - \beta_k \hat v_k - \hat x_{k+1} } \\
	&\leq \frac{4\sqrt{r}}{n}\normfro{\X_k-\bar\X_k}^2 + \normfro{\hat x_k - \beta_k \hat v_k - \hat x_{k+1} },
	\end{aligned}\ee
 where we use  $\normfro{\X_k-\bar\X_k}^2 \geq  \normfro{\X_{k+1}-\bar\X_{k+1}}^2$ in the last inequality.

 For the second term, 	since $v_{i,k}\in\T_{x_{i,k}}\M$ we have
 \be\label{ineq:derivation_lem_conv_rgd_03} 
 \begin{aligned}
 &\quad\normfro{\hat x_k - \beta_k \hat v_k - \hat x_{k+1} }\leq \frac{1}{n}\sum_{i=1}^{n}\normfro{ x_{i,k} -\alpha \grad \varphi_i^t(\X_{k}) - \beta_k v_{i,k} - x_{i,k+1} } + \frac{\alpha}{n} \normfro{\sum_{i=1}^n \grad \varphi_i^t(\X_{k})}\\
 &\stackrel{\eqref{ineq:ret_second-order}}{\leq}\frac{M}{n}\sum_{i=1}^{n}\normfro{  \alpha \grad \varphi_i^t(\X_{k})+ \beta_k v_{i,k}   }^2 + \frac{\alpha}{n} \normfro{\sum_{i=1}^n \grad \varphi_i^t(\X_{k})}\\
 &\stackrel{\eqref{ineq:bound-of-sum-gradh}}{\leq}\frac{2M\alpha^2}{n} \normfro{   \grad \varphi^t(\X_{k}) }^2 + \frac{2M \beta_k^2}{n}\normfro{\V_k}^2 + \frac{L_t\alpha}{n} \normfro{\X_k-\bar\X_k}^2\\
  &\stackrel{\eqref{ineq:bound-of-gradh}}{\leq}\frac{2L_t^2M\alpha^2 + L_t\alpha}{n} \normfro{\X_k-\bar\X_k}^2  + \frac{2M \beta_k^2}{n}\normfro{\V_k}^2\\
 & \leq \frac{10 \alpha}{n} \normfro{\X_k-\bar\X_k}^2  + \frac{2M \beta_k^2}{n}\normfro{\V_k}^2,
 \end{aligned}\ee
 where we use $\alpha\leq \frac{1}{M}$  and $L_t\leq 2$ in the last inequality. Plugging \eqref{ineq:derivation_lem_conv_rgd_03} into \eqref{ineq:derivation_lem_conv_rgd_02} yields
  \be\label{ineq:derivation_lem_conv_rgd_04} 
   \normfro{\bar x_{k+1} - \bar x_k + \beta_k \hat v_k}^2 \leq 2(\frac{4\sqrt{r} + 10 \alpha}{n})^2 \normfro{\X_k-\bar\X_k}^4  + 2(\frac{2M \beta_k^2}{n})^2\normfro{\V_k}^4.
   \ee
   Then, using Jensen's inequality and $\normfro{\V_k}^2\leq nD^2$ implies
   \begin{align*}
   	a_2 \leq \E_k[\normfro{\bar x_{k+1} - \bar x_k + \beta_k \hat v_k }^2]\leq 2(\frac{4\sqrt{r} + 10\alpha}{n})^2 \normfro{\X_k-\bar\X_k}^4  + 8 M^2\beta_k^4 D^4.
   \end{align*}

   	Thirdly, invoking \cref{lem:bound_of_k_k+1} yields
   	\begin{align*}
   	\normfro{\bar x_k - \bar{ x}_{k+1}}
   	 \leq \frac{1}{1-2 \delta_1^2}\left[ \frac{10 \alpha}{n} \normfro{\X_k-\bar\X_k}^2 +    2M \beta_k^2 D^2 + \beta_k\normfro{\hat v_k}\right].
   	\end{align*}
   	Hence, it follows that
   	\begin{align*}
   		&\quad a_3\leq  \frac{2}{(1-2\delta_1^2)^2}\left[ \frac{10 \alpha}{n} \normfro{\X_k-\bar\X_k}^2 +    2M \beta_k^2 D^2\right]^2  + \frac{2}{(1-2\delta_1^2)^2} \beta_k^2\E_k\normfro{\hat v_k}^2\\
   		&= \frac{2}{(1-2\delta_1^2)^2}\left[ \frac{10 \alpha}{n} \normfro{\X_k-\bar\X_k}^2 +   2M \beta_k^2 D^2 \right]^2 + \frac{2}{(1-2\delta_1^2)^2} \beta_k^2\E_k\normfro{\hat v_k - \hat g_k}^2 +  \frac{2}{(1-2\delta_1^2)^2} \beta_k^2\normfro{\hat g_k}^2\\
   		& \stackrel{(i)}{=} \frac{2}{(1-2\delta_1^2)^2}\left[ \frac{10 \alpha}{n} \normfro{\X_k-\bar\X_k}^2 +    2M \beta_k^2 D^2 \right]^2 + \frac{2}{(1-2\delta_1^2)^2 n^2} \beta_k^2\sum_{i=1}^n\E_k\normfro{  v_{i,k} -   g_{i,k}}^2+  \frac{2}{(1-2\delta_1^2)^2} \beta_k^2\normfro{\hat g_k}^2\\
   		&\stackrel{(ii)}{\leq}  \frac{4}{(1-2\delta_1^2)^2}\left[ \frac{100 \alpha^2}{n^2} \normfro{\X_k-\bar\X_k}^4 +   4M^2 \beta_k^4 D^4 \right]  + \frac{2}{(1-2\delta_1^2)^2 n} \beta_k^2\Xi^2+  \frac{2}{(1-2\delta_1^2)^2} \beta_k^2\normfro{\hat g_k}^2,   	\end{align*}	
where   (i) and (ii) hold by  the independence of $v_{i,k}$ and bounded variance of \cref{assum:stochastic-grad}, respectively.
   Therefore, by combining $a_1,a_2,a_3$ with \eqref{ineq:derivation_lem_conv_rgd_01} implies that
   \begin{align*}
   	&\quad \E_kf(\bar x_{k+1})\leq  f(\bar{x}_k)  - \frac{\beta_k}{2}\normfro{\hat g_k}^2 - \frac{\beta_k}{4}\normfro{\grad f(\bar x_k)}^2 + \frac{\beta_{k}}{2} {a_1}
   +\frac{1}{ \beta_k} {a_2}+  \frac{L_g}{2} {a_3}\\
   &\leq  f(\bar{x}_k)  - (\frac{\beta_k}{2} -  \frac{L_g  \beta_k^2}{(1-2\delta_1^2)^2})\normfro{\hat g_k}^2 - \frac{\beta_k}{4}\normfro{\grad f(\bar x_k)}^2 + \frac{\beta_k L_G^2}{2n}\normfro{\X_k-\bar\X_k}^2 + \frac{2}{\beta_k}(\frac{4\sqrt{r} + 10 \alpha}{n})^2 \normfro{\X_k-\bar\X_k}^4  \\
   &\quad+ 8 M^2\beta_k^3 D^4 +  \frac{2L_g}{(1-2\delta_1^2)^2}\left[ \frac{100 \alpha^2}{n^2} \normfro{\X_k-\bar\X_k}^4 +   4M^2 \beta_k^4 D^4 \right]  + \frac{L_g}{(1-2\delta_1^2)^2 n} \beta_k^2\Xi^2.
   \end{align*}
   By \cref{coro:rate_of_consensus}, we have $\normfro{\X_k-\bar\X_k}^2\leq nCD^2\beta_k^2$. It follows that
   \begin{align*}
   &\quad \E_kf(\bar x_{k+1})\\
   &\leq  f(\bar{x}_k)  - (\frac{\beta_k}{2} -  \frac{L_g  \beta_k^2}{(1-2\delta_1^2)^2})\normfro{\hat g_k}^2 - \frac{\beta_k}{4}\normfro{\grad f(\bar x_k)}^2\\&\quad  + \frac{L_g\Xi^2}{(1-2\delta_1^2)^2 n} \beta_k^2 + \left[ \frac{CD^2 L_G^2}{2}   +  \left(2(4\sqrt{r} + 10\alpha)^2 C^2+ 8 M^2\right)  D^4\right] \beta_k^3\\
   &\quad +  \frac{2L_g}{(1-2\delta_1^2)^2}\left[  100 \alpha^2C^2D^4+   4M^2  D^4 \right]\beta_k^4 \\
   &\leq f(\bar{x}_k)  - \frac{\beta_k }{4} \normfro{\hat g_k}^2  - \frac{\beta_k}{4}\normfro{\grad f(\bar x_k)}^2\\&\quad + \frac{3L_g\Xi^2}{2 n} \beta_k^2 + \left[ \frac{CD^2 L_G^2}{2}   +  \left(2(4\sqrt{r} + 10 \alpha)^2 C^2+ 8 M^2\right)  D^4\right] \beta_k^3\\
   &\quad +  \left( 201\alpha^2C^2D^4+  9M^2  D^4\right) L_g\beta_k^4,
   \end{align*}
   where we use $\frac{1}{(1-2\delta_1^2)^2}\leq 1.002$ and $\beta_k\leq \frac{1}{5L_g}$ in the last inequality.
The proof is completed.
\end{proof}

\begin{proof}[\textbf{Proof of \cref{thm:convergence of alg 2}}]
	Using \eqref{ineq:main-recursion} implies
	\be \bad
&\quad \E_kf(\bar x_{k+1})\\
&\leq  f(\bar{x}_k) - \frac{\beta_k}{4}\normfro{\grad f(\bar x_k)}^2+ \frac{3L_g\Xi^2}{2 n} \beta_k^2 +   (\frac{CD^2 L_g^2}{2} + \cT_1 D^4)\beta_k^3   +  \cT_2 L_gD^4\beta_k^4,
\ead\ee
 
	Taking the expectation on all $k$ and 
	telescoping the right hand side give us for any $K>0$ 
	\begin{align*}
	\sum_{k=0}^{K}\frac{\beta_k}{4} \E\normfro{\grad f(\bar x_k)}^2 \leq   f(\bar x_0) - f^* +  \frac{3L_g\Xi^2}{2 n} \sum_{k=0}^{K}\beta_k^2 +   (\frac{CD^2 L_g^2}{2} + \cT_1 D^4)\sum_{k=0}^{K}\beta_k^3    + \cT_2 L_gD^4\sum_{k=0}^{K}\beta_k^4,
	\end{align*}
	where $f^* = \min_{x\in\St(d,r)} f(x)$. Dividing both sides by $\sum_{k=0}^{K}\frac{\beta_k}{4}$ yields 
	\begin{align*}
	\min_{k=0,\ldots,K} \E\normfro{\grad f(\bar x_k)}^2 \leq  \frac{ f(\bar x_0) - f^* +  \frac{3L_g\Xi^2}{2 n} \sum_{k=0}^{K}\beta_k^2 +   (\frac{CD^2 L_g^2}{2} + \cT_1 D^4)\sum_{k=0}^{K}\beta_k^3    + \cT_2 L_gD^4\sum_{k=0}^{K}\beta_k^4}{ 	\sum_{k=0}^{K}\frac{\beta_k}{4}}.
	\end{align*}
   Let $\tilde \beta = \min\{1/L_g,\frac{1-\rho_t}{D}\}$. 
	Noticing that $\beta_k= \mathcal{O}( \min\{\frac{1-\rho_t}{D}, \frac{1}{L_G}\}\cdot\frac{1}{k} )$, $ \frac{\sum_{k=0}^{K}\beta_k^2}{ \sum_{k=0}^{K}\beta_k} = \mathcal{O}( \tilde \beta \frac{\ln(K+1))}{   \sqrt{K+1}})$,   $\frac{\sum_{k=0}^{K}\beta_k^3}{ \sum_{k=0}^{K}\beta_k}= \mathcal{O}(  \frac{\tilde \beta^2}{   \sqrt{K+1}})$ and  $ \frac{\sum_{k=0}^{K}\beta_k^4}{ \sum_{k=0}^{K}\beta_k}= \mathcal{O}(  \frac{\tilde \beta^3}{   \sqrt{K+1}})$. The proof is completed.
	\end{proof}
	The following corollary follows \cite{lian2017can}, in which the convergence results of constant stepsize $\beta_k$ is given.
	\begin{corollary}\label{coro:DRSGD constant stepsize}
	Under \cref{assump:lips in Rn,assump:doubly-stochastic,assum:stochastic-grad,assump:stepsize_beta}, suppose $\X_k\in\N$,   $t\geq \lceil\log_{\sigma_2}(\frac{1}{2\sqrt{n}})\rceil$,  $0<\alpha\leq\bar\alpha$.  
	If constant stepsize $\beta_k\equiv\beta=\frac{1}{2L_G+\Xi \sqrt{(K+1)/n}}$, where
\[ K+1\geq \max\{ \frac{n}{\Xi^2} (\max\{3L_G, \frac{5D}{\alpha \delta_1},  \frac{D\delta_1}{1-\rho_t}\})^2, \frac{n^3}{\Xi^6}\left( \frac{ CD^2 L_g^2 + (2\cT_1 +  \cT_2)  D^4}{2(f(\bar x_0) - f^*)+ 3L_G} \right)^2\} ,\]
if follows that 
	\begin{align*}
	 \min_{k=0,\ldots,K} \E\normfro{\grad f(\bar x_k)}^2\leq   \frac{ 8L_G(f(\bar x_0) - f^*)}{K+1} +\frac{ 8(f(\bar x_0) - f^*+ \frac{3L_G}{2})\Xi}{\sqrt{n(K+1)}}.
	\end{align*}
	\end{corollary}
	\begin{proof}
	Since $K+1\geq \frac{n}{\Xi^2} (\max\{3L_G, \frac{5D}{\alpha \delta_1}, \frac{D\delta_1}{1-\rho_t}\})^2$, we have 
	\[ \beta_k \leq \min\{\frac{1}{5L_G}, \frac{\alpha\delta_1}{5D},\frac{1-\rho_t}{D}\delta_1\}\] for all $k=0,1,\ldots,K$. Therefore, it follows that $\X_k\in\N$ for $k=0,1,\ldots,K$.
 Using \cref{thm:convergence of alg 2}, we have 
		\begin{align}
	&\quad \min_{k=0,\ldots,K} \E\normfro{\grad f(\bar x_k)}^2\notag \\
	&\leq  \frac{ 4(f(\bar x_0) - f^*)}{(K+1)\beta} +\frac{6L_g \beta \Xi^2}{n}+  (2 CD^2 L_g^2 + 4\cT_1 D^4)\beta^2    + 4\cT_2 L_gD^4 \beta^3\notag\\
	&\leq \frac{ 8L_G(f(\bar x_0) - f^*)}{K+1} +\frac{ 4(f(\bar x_0) - f^*)\Xi}{\sqrt{n(K+1)}} +\frac{6L_G  \Xi^2}{2nL_G + \Xi \sqrt{n(K+1)}}+  \frac{2 CD^2 L_g^2 + (4\cT_1 + 2\cT_2) D^4}{( 2L_G+\Xi \sqrt{(K+1)/n})^2}\label{ineq:corollary-constant-0}\\
	&\leq  \frac{ 8L_G(f(\bar x_0) - f^*)}{K+1} +\frac{ 4(f(\bar x_0) - f^*+ \frac{3L_G}{2})\Xi}{\sqrt{n(K+1)}} +  \frac{2 nCD^2 L_g^2 + (4\cT_1 + 2\cT_2) n D^4}{  \Xi^2 (K+1)},\label{ineq:corollary-constant-1}
	\end{align}
	where we use $\beta\leq \frac{1}{2L_G}\leq \frac{1}{2L_g}$ in  \eqref{ineq:corollary-constant-0}. 
	
	When 
	\[ K+1\geq \frac{n^3}{\Xi^6}\left( \frac{ CD^2 L_g^2 + (2\cT_1 +  \cT_2)  D^4}{2(f(\bar x_0) - f^*)+ 3L_G} \right)^2 ,\] 
	the second term in \eqref{ineq:corollary-constant-1} is greater than the third term, we get
		\begin{align*}
	&\quad \min_{k=0,\ldots,K} \E\normfro{\grad f(\bar x_k)}^2\notag \\
	&\leq  \frac{ 8L_G(f(\bar x_0) - f^*)}{K+1} +\frac{ 8(f(\bar x_0) - f^*+ \frac{3L_G}{2})\Xi}{\sqrt{n(K+1)}},
	\end{align*}
	which completes the proof. 
\end{proof}


\section{Proofs for \cref{sec:gradient-tracking}}
In this section, we use the following notations 
{\small\[ \bG_k := \begin{bmatrix}
\grad f_1(x_{1,k}) \\ 
\vdots \\ 
\grad f_n(x_{n,k})
\end{bmatrix},\ 
\Y_k=\begin{bmatrix}
y_{1,k} \\ 
\vdots \\ 
y_{n,k}
\end{bmatrix}, \  
 \hat y_k: =  \frac{1}{n}\sum_{i=1}^n  y_{i,k},
\]}
  \[ \hat g_k: = \frac{1}{n}\sum_{i=1}^n \grad f_i(x_{i,k}), \quad  \hat\bG_k:=(\mathbf{1}_n\otimes I_n) \hat g_k. \]   
\begin{proof}[\textbf{Proof of \cref{lem:uniform bound y}}]
	We prove it by induction. 
	Let $\hat g_{-1} = \hat y_0$, one  has $\normfro{y_{i,0}}\leq D$ and  \[ \normfro{y_{i,0}-\hat g_{-1}}\leq \normfro{y_{i,0} } + \normfro{\hat g_{-1}}\leq D + \frac{1}{n}\sum_{j=1}^n \normfro{y_{j,0}}\leq 2D \] for all $i\in[n]$ by \cref{assump:lips in Rn}. Suppose  for some $k\geq 0$, it follows that $\normfro{y_{i,k}}\leq 2D+L_G$ and  $\normfro{y_{i,k}-\hat g_{k-1}}\leq   2D+L_G$.  
	
	We note that the bound of $v_i$ becomes $2D+L_G$ here since $\normfro{v_{i,k}} = \normfro{\p_{T_{x_{i,k}\M}} y_{i,k}}\leq \normfro{y_{i,k}}$. 	Following the same argument in the proof of \cref{lem:convergence_of_deviation from mean}, we get $\X_{k+1}\in\N$ since $0<\alpha\leq\min\{ \frac{\Phi}{ 2L_t}, 1,\frac{1}{M}\}$ and $0\leq\beta\leq \min\{ \frac {1-\rho_t}{ L_G+2D} \delta_1, \frac { \alpha\delta_1  }{5(L_G+2D)}\}$.

  Then, we have
	\begin{align*}
	\normfro{ y_{i,k+1} - \hat g_{k}} &= \normfro{ \sum_{j=1}^n W^t_{i,j}y_{j,k}-  \hat g_{k}  + \grad f(x_{i,k+1}) - \grad f(x_{i,k})}\\
	&=\normfro{ \sum_{j=1}^n (W^t_{i,j}-\frac{1}{n})(y_{j,k} - \hat g_{k-1}) + \grad f(x_{i,k+1}) - \grad f(x_{i,k})}\\
	&\stackrel{\eqref{ineq:lips_riemanniangrad}}{\leq} \sigma_2^t \sqrt{n} \normfro{y_{j,k} - \hat g_{k-1}} + L_G\normfro{x_{i,k+1} - x_{i,k}}\\
	&\stackrel{\eqref{ineq:ret_nonexpansive}}{\leq}   \sigma_2^t \sqrt{n} \normfro{y_{j,k} - \hat g_{k-1}} + L_G (\alpha \normfro{\grad \varphi_i^t(\X_{k})} + \beta \normfro{y_{i,k}})\\
	&\stackrel{\eqref{ineq:bound-of-gradh-i}}{\leq}  \frac{1}{2}\normfro{y_{j,k} - \hat g_{k-1}}  +2 L_G \alpha  \delta_2 + L_G\beta \normfro{y_{i,k}}\\
	&\leq \frac{1}{2 }(2D+L_G) +    2\delta_2  L_G+ \frac{L_G}{5}\delta_1\alpha\\
	&\stackrel{\eqref{delta_1_and_delta_2}}{\leq} D+ L_G.
	\end{align*}
	Hence, $\normfro{y_{i,k+1}} \leq \normfro{ y_{i,k+1} - \hat g_{k}} + \normfro{\hat g_k} \leq {L_G} + 2D$, where we use $\normfro{\hat g_k}\leq D$.  	Therefore, we get $\normfro{y_{i,k}}\leq L_g+ 2D$ for all $i,k$ and $\X_k\in \N$.

 Using the same argument of  \cref{coro:rate_of_consensus}, there exists  some $C_1=\mathcal{O}(\frac{1}{(1-\rho_t)^2})$ that is independent of  $L_G$ and $D$ such that 
 \be \frac{1}{n}\normfro{\X_k-\bar\X_k}^2\leq  C_1 (L_G + 2D)^2\beta^2, k \geq 0.\ee      The proof is completed. 
\end{proof}

Next, we present the relations between the consensus error and the gradient tracking error.  
\begin{lemma}\label{lem:error-bounds-gradient-tracking}
	Under the same conditions of \cref{lem:uniform bound y}, 
	one has the following error bounds for any $k\geq 0$:
	\begin{enumerate}
		\item Successive gradient error: \be\label{ineq: Successive gradient error} \normfro{\bG_{k+1} - \bG_k} \leq 2\alpha  L_G\normfro{\X_k - \bar\X_k} +  \beta L_G\normfro{\Y_k}. \ee
		\item Successive tracking error:   \be\label{ineq: Successive tracking error} \normfro{ \Y_{k+1} - \hat \bG_{k+1} }  \leq  \sigma_2^t \normfro{\Y_k - \hat \bG_k}+  \normfro{\bG_{k+1}- \bG_k }. \ee	
		\item Successive consensus error: for  $\rho_t=\sqrt{1-\gamma_t\alpha }\in(0,1)$, 
		\be\label{ineq:Successive consensus error}
		\normfro{\X_{k+1}-\bar{\X}_{k+1}}\leq \rho_t\normfro{\X_k - \bar \X_k} + \beta \normfro{\Y_k}.
	  \ee
		\item Associating $\Y_k, \hat\bG_k$ with above items: 
		\be\label{ineq:Bound-of-Y} \normfro{\Y_{k}} \leq   \normfro{\Y_{k} -\hat\bG_{k}} + \normfro{\hat\bG_{k}} . \ee
	\end{enumerate}
\end{lemma}

\begin{proof}[\textbf{Proof of \cref{lem:error-bounds-gradient-tracking}}]
	By \cref{lem:uniform bound y}, we know $\X_k\in\N$ for all $k\geq 0$. 
	\begin{enumerate}
		\item Using \cref{lem:lipschitz} yields
		\begin{align*}
		\normfro{\bG_{k+1} -\bG_k}\leq L_G\normfro{\X_{k+1} - \X_{k}}.
		\end{align*}
		By   \cref{lem:nonexpansive_bound_retraction}, it follows that
		\[	\normfro{\X_k-\X_{k+1}}\leq \alpha\normfro{\grad \varphi^t(\X_k)}+ \beta \normfro{\V_k}\stackrel{\eqref{ineq:bound-of-gradh}}{\leq}  2\alpha\normfro{\X_k - \bar\X_k} + \beta\normfro{\Y_k},\] where we use $\normfro{\V_k} \leq \normfro{\Y_k}$.  Hence, the inequality \eqref{ineq: Successive gradient error} is proved. 
		\item Denote $J = \frac{1}{n}\mathbf{1}_n\mathbf{1}_n^\top$. 
		Note that 
		\begin{align*}
		\Y_{k+1} - \hat \bG_{k+1} &= ((I_n-J)\otimes I_n)\Y_{k+1} \\
		&= ((I_n-J)\otimes I_n)\left [ (W^t\otimes I_n)\Y_k  + \bG_{k+1} -\bG_k \right]\\
		& =   ((W^t-J)\otimes I_n)\Y_k  +  ((I_n-J)\otimes I_n) (\bG_{k+1} -\bG_k )
		\end{align*}
		where we use $((I_n-J)\otimes I_n) (W^t\otimes I_n) = (W^t-J)\otimes I_n$. 
		It follows that
		\begin{align*}
		\normfro{\Y_{k+1} - \hat \bG_{k+1}}   \leq \sigma_2^t\normfro{\Y_k - \hat \bG_k}+ \normfro{\bG_{k+1} -\bG_k}
		\end{align*}
		\item Note that $\normfro{\V_k}\leq \normfro{\Y_k}$. Then the desired result follows the same line as that of \cref{lem:recursive lemma}.
		\item  This follows   from the triangle inequality. 
	\end{enumerate}
\end{proof}

To show \cref{thm:theorem-convergence-grad-tracking}, we firstly show a descent lemma.  Note that an extra $\normfro{\hat \bG_k}^2 = n\normfro{\hat g_k}^2$ appears  in \eqref{ineq:Bound-of-Y},   what is  we aim at bounding in the optimization problem \eqref{opt_problem}.  By combining with the following lemmas, we can quickly obtain the final convergence result. 
\begin{lemma}\label{lem:descent-grad-tracking}
	Under the same conditions of \cref{lem:uniform bound y}, it follows that 
	\be\label{ineq:main-gradient-tracking}
	\bad
	f(\bar x_{k+1})&\leq f(\bar{x}_k)  -(\beta -  4L_G \beta^2 ) \normfro{\hat g_k}^2 + \G_0 \frac{ L_G}{n} \normfro{\X_{k+1}-\bar\X_{k+1}}^2  + \G_1 \frac{ L_G}{n} \normfro{\X_k-\bar\X_k}^2    + \G_2\frac{ L_G}{n}\beta^2\normfro{\Y_k}^2  ,
	\ead
	\ee 
	where $\G_0=\frac{4r(L_G + 2D)^2C_1}{L_G^2} ,  $	$\G_1=  1+\G_0  + \frac{2 D\alpha +8MD\alpha^2}{L_G} +13C_1 \delta_1^2 \alpha^4 $, $\G_2 = \frac{2MD}{L_G} + \frac{\delta_1^2}{2}+  5$
	and  $C_1$ is given in \cref{lem:uniform bound y}. 
\end{lemma}
Since $D=\max_{x\in\St(d,r)}\normfro{\nabla f(x)}\leq \sqrt{r}\cdot \max_{x\in\St(d,r)}\normtwo{\nabla f(x)}=\sqrt{r}L_n.$
By the choice of $\alpha$, the constants in \cref{lem:descent-grad-tracking} are given by $\G_0 = \mathcal{O}(r^2C_1)$ , $\G_1 = \mathcal{O}(r^2C_1)$ and $\G_2 = \mathcal{O}(M)$. 

\begin{proof}[\textbf{Proof of \cref{lem:descent-grad-tracking}}] 
It follows from \cref{lem:lipschitz} that
\be\label{ineq:grad_tracking_der-0}  \normfro{\hat{g}_k -\grad f(\bar x_k)  }^2 \leq \frac{1}{n} \sum_{i=1}^n \normfro{\grad f_i(x_{i,k}) - \grad f(\bar x_k)}^2\leq \frac{L_G^2}{n}\normfro{\X_k - \bar \X_k}^2.\ee
	By invoking \cref{lem:lipschitz} and noting $L_g\leq L_G$, we also have
\be\label{ineq:grad_tracking_der-01}\begin{aligned}
f(\bar x_{k+1}) &\leq f(\bar{x}_k) +  \inp{\grad f(\bar x_k)}{\bar x_{k+1} - \bar x_k} +  \frac{L_g}{2}\normfro{\bar x_{k+1} - \bar x_k}^2\\ 
&\leq  f(\bar{x}_k)+ \inp{\hat g_k}{  \bar x_{k+1} - \bar x_k} + \inp{\grad f(\bar x_k)-\hat g_k}{ \bar x_{k+1} - \bar x_k} +  \frac{L_G}{2} \normfro{\bar x_{k+1} - \bar x_k}^2\\
&\leq f(\bar{x}_k)+ \inp{\hat g_k}{\bar x_{k+1} - \bar x_k} + \frac{1}{L_G}  \normfro{\grad f(\bar x_k)-\hat g_k}^2 +   \frac{3L_G}{4} \normfro{\bar x_{k+1} - \bar x_k}^2\\
&\stackrel{\eqref{ineq:grad_tracking_der-0}}{\leq } f(\bar{x}_k) + \inp{\hat g_k}{\bar x_{k+1} - \bar x_k}+ \frac{ L_G}{ n}\normfro{\X_k-\bar\X_k}^2 +    \frac{3L_G}{4}  \normfro{\bar x_{k+1} - \bar x_k}^2\\
&=  f(\bar{x}_k) + \inp{\hat g_k}{\hat x_{k+1}  - \hat x_k} +\inp{\hat g_k}{\bar x_{k+1} -\hat x_{k+1} +\hat x_k - \bar x_k}  
+ \frac{ L_G}{ n}\normfro{\X_k-\bar\X_k}^2 +   \frac{3L_G}{4}  \normfro{\bar x_{k+1} - \bar x_k}^2.
\end{aligned}\ee
	
	Note that for $\beta>0$, we have
	\begin{align*}
	&\quad  \inp{\hat g_k}{\bar x_{k+1} -\hat x_{k+1} +\hat x_k - \bar x_k}  \leq  \frac{\beta ^2 L_G}{2}\normfro{\hat g_k}^2+ \frac{1}{ \beta ^2 L_G}\normfro{\hat x_{k}  - \bar x_k}^2+   \frac{1}{\beta ^2 L_G}\normfro{ \bar x_{k+1} - \hat x_{k+1}}^2.
	\end{align*} 
	 
	Plugging this into \eqref{ineq:grad_tracking_der-01} yields
		\be\label{ineq:grad_tracking_der-1} 
	\begin{aligned}
		&\quad f(\bar x_{k+1})\\
		&\leq  f(\bar{x}_k) + \underbrace{\inp{\hat g_k}{\hat x_{k+1} -\hat x_{k} } }_{:=b_1} + \frac{\beta ^2 L_G}{2}\normfro{\hat g_k}^2+ \underbrace{\frac{1}{ \beta ^2 L_G}( \normfro{\hat x_{k}  - \bar x_k}^2+    \normfro{ \bar x_{k+1} - \hat x_{k+1}}^2)}_{:=b_2}  \\
		&\quad + \frac{ L_G}{ n}\normfro{\X_k-\bar\X_k}^2 +  \underbrace{  \frac{3L_G}{4}  \normfro{\bar x_{k+1} - \bar x_k}^2}_{:=b_3}.
	\end{aligned}
	\ee
	Firstly, we have
	\be\label{ineq:grad_tracking_der-3-1} 
	\begin{aligned}	 
		b_1& =  \inp{  \hat g_k}{\hat x_{k+1} - \hat x_k - \beta \hat g_k + \beta \hat g_k}   \\
		& = -\beta \normfro{\hat g_k}^2 + \inp{\hat g_k}{\frac{1}{n}\sum_{i=1}^n[x_{i,k+1} - (x_{i,k}-\beta v_{i,k} - \alpha \grad \varphi_i^t(\X_{k})) ]} \\
		&\quad + \inp{\hat g_k}{  \frac{1}{n}\sum_{i=1}^n [\beta(y_{i,k} -    v_{i,k}) - \alpha \grad \varphi_i^t(\X_{k})] }.
	\end{aligned}\ee
	 
Since $y_{i,k} - v_{i,k} \in N_{x_{i,k}}\M$, it follows that
\begin{align*} 
&\quad \inp{\hat g_k}{  \frac{\beta}{n}\sum_{i=1}^n(y_{i,k} -    v_{i,k}) - \alpha \grad \varphi_i^t(\X_{k}) } \\
& \stackrel{\eqref{ineq:bound-of-sum-gradh}}{\leq}\frac{\beta}{n} \sum_{i=1}^n \inp{\hat g_k - \grad f_i(x_{i,k})}{y_{i,k} -    v_{i,k}}  +  \frac{2\alpha}{n} \normfro{\hat g_k}\cdot \normfro{\X_k - \bar \X_k}^2  \\
& \leq  \frac{1}{4nL_G} \sum_{i=1}^n \normfro{\hat g_k - \grad  f_i(x_{i,k})}^2 + \frac{\beta^2L_G}{n}\sum_{i=1}^n\normfro{\p_{N_{x_{i,k}}} y_{i,k}}^2 + \frac{2\alpha D}{n}   \normfro{\X_k - \bar \X_k}^2\\
&\leq  \frac{1}{4n^2L_G} \sum_{i=1}^n \sum_{j=1}^n\normfro{\grad f_j(x_{j,k}) - \grad  f_i(x_{i,k})}^2 + \frac{\beta^2L_G}{n}\normfro{\Y_k}^2 + \frac{2\alpha D}{n}   \normfro{\X_k - \bar \X_k}^2\\
&\leq \frac{L_G+2\alpha D}{n}\normfro{\X_k - \bar \X_k}^2 + \frac{\beta^2L_G}{n} \normfro{\Y_k}^2,
\end{align*}
where we use \cref{lem:lipschitz} in the last inequality. This, together with \eqref{ineq:grad_tracking_der-3-1} and \eqref{ineq:bound-of-sum-gradh} implies
\be\label{ineq:grad_tracking_der-3} 
\begin{aligned}
 b_1	&\leq  -\beta \normfro{\hat g_k}^2 + \inp{\hat g_k}{\frac{1}{n}\sum_{i=1}^n[x_{i,k+1} - (x_{i,k}-\beta v_{i,k} - \alpha \grad \varphi_i^t(\X_{k})) ]} + \frac{L_G + 2D\alpha}{n}\normfro{\X_k - \bar \X_k}^2 + \frac{\beta^2L_G}{n} \normfro{\Y_k}^2  \\
	&\leq  -\beta \normfro{\hat g_k}^2 +\frac{D}{n}\sum_{i=1}^{n}\normfro{ x_{i,k} - \alpha \grad \varphi_i^t(\X_{k}) - \beta v_{i,k} - x_{i,k+1} }+   \frac{L_G  + 2D\alpha}{n}\normfro{\X_k - \bar \X_k}^2 + \frac{\beta^2L_G}{n} \normfro{\Y_k}^2  \\
	&\stackrel{\eqref{ineq:ret_second-order}}{\leq} -\beta \normfro{\hat g_k}^2 + \frac{MD}{n}\sum_{i=1}^{n}\normfro{  \alpha \grad \varphi_i^t(\X_{k}) + \beta v_{i,k}   }^2 +   \frac{L_G + 2D\alpha}{n}\normfro{\X_k - \bar \X_k}^2 + \frac{\beta^2L_G}{n} \normfro{\Y_k}^2 \\
	&\stackrel{\eqref{ineq:bound-of-sum-gradh}}{\leq}  -\beta \normfro{\hat g_k}^2 +\frac{2MD\alpha^2}{n} \normfro{   \grad \varphi^t(\X_{k}) }^2 + \frac{2MD \beta^2}{n}\normfro{\Y_k}^2+   \frac{L_G  + 2D\alpha}{n}\normfro{\X_k - \bar \X_k}^2 + \frac{\beta^2L_G}{n} \normfro{\Y_k}^2\\
	&\stackrel{\eqref{ineq:bound-of-gradh}}{\leq} -\beta \normfro{\hat g_k}^2 + \frac{8MD\alpha^2 + 2 D\alpha + L_G}{n} \normfro{\X_k-\bar\X_k}^2  + \frac{(2M D+L_G)\beta^2}{n}\normfro{\Y_k}^2,
\end{aligned}\ee
where we use $\normfro{\hat g_k} \leq D$. 
 
	Secondly, we use the following inequality to derive the upper   bound of $b_2$.
	From \cref{lem:uniform bound y}, we have $\X_{k+1}\in\N$. One has
	\be\label{ineq:grad_tracking_der-2} 
	\begin{aligned}
	    &\quad  \normfro{\bar x_k - \hat x_k}^2 + \normfro{\bar x_{k+1} - \hat x_{k+1}}^2  \\
		&\stackrel{\eqref{key}}{\leq }\frac{4r}{n^2}(\normfro{\X_k-\bar\X_k}^4 + \normfro{\X_{k+1}-\bar\X_{k+1}}^4).
	\end{aligned}\ee
	We then obtain
		
	\begin{align*}
	  b_2 \leq\frac{4r}{n^2\beta^2L_G}(\normfro{\X_k-\bar\X_k}^4 + \normfro{\X_{k+1}-\bar\X_{k+1}}^4).
	\end{align*} 

	 Thirdly, invoking \cref{lem:bound_of_k_k+1} and $\alpha\leq 1/M$ yields
	 \begin{align*}
	 \normfro{\bar x_k - \bar{ x}_{k+1}}
	 &\leq \frac{1}{1-2\delta_1^2}\left[ \frac{10 \alpha}{n} \normfro{\X_k-\bar\X_k}^2 +   \frac{ 2M \beta^2 }{n} \normfro{\Y_k}^2 + \beta\normfro{\hat v_k}\right].
	 \end{align*}
	 Then, it  follows from $ \beta \normfro{\Y_k}\leq \frac{ \alpha \delta_1}{5 } $ that  
	 \begin{align*}
	 &\quad b_3\leq \frac{3L_G}{4}\left(  \frac{2}{(1-2\delta_1^2)^2}\left[ \frac{10 \alpha}{n} \normfro{\X_k-\bar\X_k}^2 +  \frac{ 2M \beta^2 }{n} \normfro{\Y_k}^2   \right]^2  + \frac{2}{(1-2\delta_1^2)^2} \beta^2 \normfro{\hat v_k}^2\right) \\
	 &\leq  \frac{3L_G}{(1-2\delta_1^2)^2}\left[ \frac{100 \alpha^2}{n^2} \normfro{\X_k-\bar\X_k}^4 +  \frac{(  M\alpha \delta_1\beta)^2 }{10n}\normfro{\Y_k}^2  \right]  + \frac{3L_G}{(1-2\delta_1^2)^2} \beta^2 (\normfro{\hat y_k}^2+ \normfro{\hat v_k-\hat y_k}^2)\\
	 &\leq  \frac{3L_G}{(1-2\delta_1^2)^2}\left[ \frac{100\alpha^2}{n^2} \normfro{\X_k-\bar\X_k}^4 +  \frac{(  M\alpha \delta_1\beta)^2 }{10n}\normfro{\Y_k}^2  \right]  + \frac{3L_G}{(1-2\delta_1^2)^2} \beta^2 (\normfro{\hat g_k}^2+ \frac{1}{n}\normfro{\Y_k}^2),
	 \end{align*}
    where we use $\hat y_k = \hat g_k$ and $\normfro{\hat v_k - \hat y_k}^2\leq \frac{1}{n}\normfro{\p_{N_{x_{i,k}}} y_{i,k}}^2 \leq \frac{1}{n}\normfro{\Y_k}^2.$ It follows from \eqref{ienq:bound_x_gt} that
    \[ \normfro{\X_k - \bar\X_k}^2 \leq  C_1(L_G+2D)^2\beta^2\leq \frac{C_1\alpha^2\delta_1^2}{25}, \] 
     where we use $\beta\leq \frac{\alpha \delta_1}{5(L_G + 2D)}$.
     Therefore, we get
     \be\label{ineq:grad_tracking_der-2-1} \begin{aligned}
    b_2\leq \frac{4r(L_G+ 2D)^2C_1 }{n L_G  }(\normfro{\X_k-\bar\X_k}^2 + \normfro{\X_{k+1}-\bar\X_{k+1}}^2).
    \end{aligned}\ee
    and 
    \be\label{ineq:grad_tracking_der-2-3}\begin{aligned}
    	b_3&\leq  \frac{3L_G}{(1-2\delta_1^2)^2}\left[ \frac{4C_1\delta_1^2 \alpha^4}{n } \normfro{\X_k-\bar\X_k}^2 +  \frac{ \delta_1^2 }{10n}\beta^2\normfro{\Y_k}^2  \right]  + \frac{3L_G}{(1-2\delta_1^2)^2} \beta^2 (\normfro{\hat g_k}^2+ \frac{1}{n}\normfro{\Y_k}^2)\\
    	&\leq \frac{13L_G C_1  \delta_1^2 \alpha^4}{n}\normfro{\X_k-\bar\X_k}^2 +\frac{7}{2} L_G \beta^2 \normfro{\hat g_k}^2 +   \frac{  \frac{\delta_1^2}{2}+  4  }{ n}L_G \beta^2\normfro{\Y_k}^2,
    \end{aligned}\ee
    where we use  $\alpha \leq \frac{1}{M}$ and $\frac{1}{(1-2\delta_1^2)^2}\leq 1.002$. 
	 Therefore, by combining the upper bound of $b_1,b_2,b_3$ with \eqref{ineq:grad_tracking_der-1} implies
	 \begin{align*}
	 &\quad f(\bar x_{k+1})\leq f(\bar{x}_k) + {b_1} + \frac{\beta ^2 L_G}{2}\normfro{\hat g_k}^2+  {b_2}   + \frac{ L_G}{ n}\normfro{\X_k-\bar\X_k}^2 +  {b_3}\\
	 &\leq  f(\bar{x}_k)  -(\beta -  4L_G \beta^2 ) \normfro{\hat g_k}^2 + \frac{ L_G+ \frac{4r(L_G + 2D)^2C_1}{L_G}  + 2 D\alpha +8MD\alpha^2 +13L_G C_1 \delta_1^2 \alpha^4 }{n} \normfro{\X_k-\bar\X_k}^2  \\
	 &\quad + \frac{4r(L_G + 2D)^2 C_1 }{n L_G } \normfro{\X_{k+1}-\bar\X_{k+1}}^2  + \frac{2M D+ ( \frac{\delta_1^2}{2}+  5)L_G}{n}\beta^2\normfro{\Y_k}^2  .
	 \end{align*} 
	 The proof is completed.

\end{proof}

To proceed, we need the following recursive lemma, which is helpful to combine \cref{lem:error-bounds-gradient-tracking} and \cref{lem:descent-grad-tracking}.  It is a little different from the original one in \cite{xu2015augmented}. We only change   $\sqrt{\sum_{l=0}^k u_i^2}$ and $\sqrt{ \sum_{l=0}^k w_i^2}$ to be $ {\sum_{l=0}^k u_i^2}$ and $ { \sum_{l=0}^k w_i^2}$.  
\begin{lemma}\cite[Lemma 2]{xu2015augmented}\label{lem:gt-recursion-lem}
	Let $\{u_k\}_{k\geq 0}$ and $\{w_k\}_{k\geq 0}$ be two positive scalar sequences such that for all $k\geq 0$ 
	\[  u_{k+1}\leq \eta u_{k} + w_k,  \]
	where $\eta\in(0,1)$ is the decaying factor. Let $\Gamma(k) = {\sum_{l=0}^k u_i^2}$ and   $\Omega(k) = {\sum_{l=0}^k w_i^2}.$ Then we have
	\[ \Gamma(k) \leq c_0 \Omega(k) + c_1, \]
	where $c_0 = \frac{ {2}}{(1-\eta)^2}$ and $c_1 =  \frac{2}{1-\eta^2} u^2_0.$
\end{lemma}

\begin{proof}[\textbf{Proof of \cref{thm:theorem-convergence-grad-tracking}}]
	
	Applying \cref{lem:gt-recursion-lem} to \eqref{ineq:Successive consensus error} yields
	\begin{align}\label{ineq:sum_sucessive_consensus-error}
	\frac{1}{n}\sum_{k=0}^{K}\normfro{\X_k - \bar\X_k}^2\leq  \tilde{C}_0 \cdot	\frac{\beta^2}{n} \sum_{k=0}^{K}\normfro{\Y_k}^2 + \tilde{C}_1,
	\end{align}
	where $\tilde{C}_0 = \frac{2}{(1-\rho_t)^2}$ and $\tilde{C}_1 = \frac{2}{1-\rho_t^2} 	\frac{1}{n}\normfro{\X_0 - \bar \X_0}^2$. 

	It follows from \cref{lem:descent-grad-tracking} that
	\be\label{ineq:thm2-eq0}
	\begin{aligned}
		&\quad  f(\bar x_{K+1}) \\
		&\leq  f(\bar{x}_0)  - (\beta -  4L_G \beta^2 ) \sum_{k=0}^{K} \normfro{\hat g_k}^2 +  \frac{\G_1 L_G}{n}\sum_{k=0}^{K} \normfro{\X_k-\bar\X_k}^2  +   \frac{\G_0 L_G}{n}\sum_{k=1}^{K+1} \normfro{\X_k-\bar\X_k}^2  +     \frac{\G_2 L_G }{ n} \beta^2  \sum_{k=0}^{K}\normfro{\Y_k}^2\\	
	&\stackrel{\eqref{ineq:sum_sucessive_consensus-error}}{\leq} f(\bar x_0) - (\beta -  4L_G \beta^2 ) \sum_{k=0}^{K} \normfro{\hat g_k}^2 
	+  (  \G_1\tilde{C}_0 +\G_0\tilde{C}_0  + \G_2  ) \frac{ L_G \beta^2}{n}\sum_{k=0}^{K}\normfro{\Y_k}^2   + \frac{\G_0\tilde{C}_0L_G\beta^2\normfro{\Y_{K+1} }^2}{n}  + \tilde{C}_1(\G_1+\G_0) L_G\\
	& \leq   f(\bar x_0) - \frac{\beta}{2}  \sum_{k=0}^{K} \normfro{\hat g_k}^2 
	+     \G_3 \frac{ L_G \beta^2}{n}\sum_{k=0}^{K}\normfro{\Y_k}^2 +\G_4 L_G,
	\end{aligned}\ee
  where  we use $\beta\leq \min\{ \frac{1}{8L_G}, \frac{\alpha\delta_1}{5(L_G+2D)}\}$,  $\beta^2\normfro{\Y_{K+1} }^2 \leq n ( L_G+2D)^2\beta^2\leq \frac{\delta_1^2\alpha^2 n}{25} $ and $\G_3: =\G_1\tilde{C}_0 +\G_0\tilde{C}_0  + \G_2  $ and $\G_4 := \frac{ \G_0\tilde{C}_0  \delta_1^2\alpha^2  }{25}  + \tilde{C}_1(\G_1+4rC_1)  $ in the last inequality.
  
	We are going to associate $\normfro{\hat g_k}^2$ with $\normfro{\Y_k}^2$.  By \eqref{ineq:Bound-of-Y}, we get
	\begin{align}\label{ienq:thm2-eq1}
	-	\sum_{k=0}^{K}\normfro{\hat g_k}^2=-\frac{1}{n}	\sum_{k=0}^{K}\normfro{\hat \bG_k}^2\leq  \frac{1}{n}\sum_{k=0}^{K}\normfro{\Y_k - \hat \bG_k}^2 - \frac{1}{2n} \sum_{k=0}^{K}\normfro{\Y_k}^2
	\end{align}
	Again, applying \cref{lem:gt-recursion-lem} to \eqref{ineq: Successive tracking error} yields
	\begin{align*}
	\frac{1}{n}\sum_{k=0}^{K}\normfro{\Y_k - \hat \bG_k}^2&\leq   \tilde{C}_2 	\frac{1}{n}\sum_{k=0}^{K}\normfro{\bG_{k+1} - \bG_{k}}^2 + \tilde{C}_3\\
	&\stackrel{\eqref{ineq: Successive gradient error}}{\leq}   \tilde{C}_2	\frac{1}{n}\sum_{k=0}^{K}(8\alpha^2 L_G^2\normfro{\X_k - \bar \X_k}^2 + 2\beta^2 L_G^2\normfro{\Y_k}^2 ) + \tilde{C}_3\\
	&\stackrel{\eqref{ineq:sum_sucessive_consensus-error}}{\leq }(8\alpha^2 \tilde{C}_0\tilde{C}_2  +  2\tilde{C}_2 )L_G^2  \beta^2 \frac{1}{n} \sum_{k=0}^{K}\normfro{\Y_k}^2 + 8\alpha^2\tilde{C}_1\tilde{C}_2L_G^2 +\tilde{C}_3\\
	&\leq (8 \tilde{C}_0   +  \frac{1}{2} \tilde{C}_2 ) \alpha \delta_1 L_G   \beta  \frac{1}{n} \sum_{k=0}^{K}\normfro{\Y_k}^2 + 8\alpha^2\tilde{C}_1\tilde{C}_2L_G^2 +\tilde{C}_3,
	\end{align*}  
	where $\tilde{C}_2 = \frac{2}{(1-\sigma_2^t)^2}$ and $\tilde{C}_3 = \frac{2}{1-\sigma_2^{2t}} \cdot	\frac{1}{n}\normfro{\Y_0 - \hat\bG_0}^2$. The last line is due to $\beta\leq \frac{\alpha\delta_1}{5L_G}$ and  $\alpha^2\tilde{C}_2\leq \tilde{C}_2\leq \frac{2}{(1-\frac{1}{2\sqrt{n}})^2}\leq  5$.  Plugging this into \eqref{ienq:thm2-eq1} implies
	\begin{align} \label{ienq:thm2-eq2}
	-	\sum_{k=0}^{K}\normfro{\hat g_k}^2&\leq  \left[(8 \tilde{C}_0   +  \frac{1}{2} \tilde{C}_2 ) \alpha \delta_1L_G   \beta- \half \right]  \frac{1}{n} \sum_{k=0}^{K}\normfro{\Y_k}^2   +  8\alpha^2\tilde{C}_1\tilde{C}_2L_G^2 +\tilde{C}_3.
	\end{align}
	Hence,  it follows from \cref{ineq:thm2-eq0} that
	\be\label{ineq:thm2-eq3}
	\begin{aligned}
		&\quad  f(\bar x_{K+1}) \\
		&\stackrel{\eqref{ienq:thm2-eq2}}{\leq} f(\bar x_0) -     \frac{ \beta}{2}   \left( \frac{1}{2} -  \left[ 2 \G_3  + (8 \tilde{C}_0   +  \frac{1}{2} \tilde{C}_2 ) \alpha \delta_1\right]L_G   \beta\right) 
		 \frac{1}{n} \sum_{k=0}^{K}\normfro{\Y_k}^2    +    \frac{ \beta}{2}  \left(8\alpha^2\tilde{C}_1\tilde{C}_2L_G^2 +\tilde{C}_3 \right)  +   \G_4 L_G\\
		&\leq f(\bar x_0) -\frac{\beta}{8} \frac{1}{n} \sum_{k=0}^{K}\normfro{\Y_k}^2   +  \frac{ \beta }{2}\left(8\alpha^2\tilde{C}_1\tilde{C}_2L_G^2 +\tilde{C}_3 \right)  +    \G_4 L_G\\
	\end{aligned}\ee
	where the last inequality is due to $\beta\leq \displaystyle \frac{1}{4L_G(2 \G_3  + (8 \tilde{C}_0   +  \frac{1}{2} \tilde{C}_2 ) \alpha \delta_1)}$.
	
	Then, we get
	\be\label{ineq:thm2-eq4}
	\begin{aligned}
	\frac{\beta }{8} \sum_{k=0}^K \normfro{\hat g_k}^2\leq	\frac{\beta }{8} \cdot \frac{1}{n} \sum_{k=0}^{K}\normfro{\Y_k}^2 \leq  f(\bar x_0) -  f^*  +   \tilde{C}_4 +    \G_4 L_G,
	\end{aligned}\ee   
	where $\tilde{C}_4 = ( 8\alpha^2\tilde{C}_1\tilde{C}_2L_G^2 +\tilde{C}_3)\frac{\beta}{2}=\mathcal{O}(\frac{r \delta_1^2 L_G}{(1-\sigma_2^t)^2})$ and $f^*=\min_{x\in\St(d,r)}f(x)$. 
	This implies
	\be\label{ineq:convergence_Y} \min_{k=0,\ldots,K}  \normfro{\hat  g_k}^2 = \min_{k=0,\ldots,K}  \normfro{\hat  y_k}^2 \leq  \min_{k=0,\ldots,K} \frac{1}{n}\normfro{\Y_k}^2\leq  \frac{8(f(\bar x_0) -  f^*  +   \tilde{C}_4 +    \G_4 L_G )}{\beta\cdot K}.  \ee
	It  then follows from \eqref{ineq:sum_sucessive_consensus-error} that
	\begin{align*} 
	\min_{k=0,\ldots,K}\frac{1}{n} \normfro{\X_k - \bar\X_k}^2\leq   \frac{8\beta(f(\bar x_0) -  f^*  +   \tilde{C}_4+    \G_4L_G  )\tilde{C}_0 + \tilde{C}_1}{  K}.
	\end{align*}
	Finally, noticing $\beta\leq \frac{\alpha \delta_1}{5L_G}$ and 
	\[ \normfro{\grad f(\bar x_k)}^2 \leq 2\normfro{\hat g_k}^2 + 2\normfro{\grad f(\bar x_k)-\hat g_k}^2 \leq 2\normfro{\hat g_k}^2 + \frac{2L_G^2}{n}\normfro{\X_k-\bar\X_k}^2. \]
	We finally have
	\begin{align*}
	\min_{k=0,\ldots,K} \normfro{\grad f(\bar x_k)}^2 \leq \frac{(16+ \alpha^2 \delta_1^2\tilde{C}_0)(f(\bar x_0) -  f^*  +   \tilde{C}_4+     \G_4 L_G ) + \tilde C_1 L_G}{\beta\cdot K}.
	\end{align*}
	The proof is completed.
\end{proof}

\end{document}